\documentclass{amsart}

\usepackage{fancybox}
\usepackage{color}
\usepackage{leftidx}
\usepackage{amsmath,amssymb,graphicx,epsfig}
\usepackage{amsmath}
\usepackage{amssymb}

 \newtheorem{thm}{Theorem}[section]
 \newtheorem{cor}[thm]{Corollary}
\newtheorem{lem}[thm]{Lemma}
 \newtheorem{prop}[thm]{Proposition}
 \newtheorem{claim}[thm]{Claim}

 \newtheorem{defn}[thm]{Definition}
 \newtheorem{rem}[thm]{Remark}
 \newtheorem{example}[thm]{Example}
 \newtheorem{assume}[thm]{ASSUMPTION}

\newcommand{\eps}{\varepsilon}

\newcommand{\Delt}{\Delta t}
\newcommand{\Delx}{\Delta x}

\newcommand \be     {\begin{equation}}
\newcommand \ee     {\end{equation}}
\newcommand {\RR} {\mathbb{R}}
\newcommand {\Rplus} {\mathbb{R}_+}
\newcommand {\Rplusc} {\overline{\mathbb{R}_+}}
\newcommand {\RplusT} {[0,T]}

\newcommand {\GRST} {\Gamma_k^{spacetime}}
\newcommand{\set}[1]{\left\{#1\right\}}

\newcommand \Gcal   {\mathcal G}

  \newcommand{\fU}{\mathfrak{U}}
  \newcommand{\suml}{\sum\limits}
  \newcommand{\dert}{\frac{\partial}{\partial t}}
 \numberwithin{equation}{section}

\begin{document}

\title [CONSISTENCY AND CONVERGENCE OF APPROXIMATE BALANCE LAWS]
{CONSISTENCY AND CONVERGENCE OF FINITE VOLUME APPROXIMATIONS TO NONLINEAR HYPERBOLIC BALANCE LAWS }

\author{Matania Ben-Artzi}
\address{Matania Ben-Artzi: Institute of Mathematics, The Hebrew University, Jerusalem 91904, Israel}
\email{mbartzi@math.huji.ac.il}
\author{Jiequan Li}
\address{Jiequan Li: Laboratory of Computational Physics,  Institute of Applied Physics and Computational Mathematics,
 Beijing , China; Center for Applied Physics and Technology, Peking University, China}
\email{li\_jiequan@iapcm.ac.cn}





\thanks{ The first author thanks the Institute of Applied Physics and Computational Mathematics, Beijing,  for the hospitality and support. The second author is supported by NSFC (nos. 11771054, 91852207) and Foundation of LCP. It is a pleasure to thank C. Dafermos, T. Gallou\"{e}t,  R. Herbin and M. Slemrod for many useful comments. We are very grateful to the two anonymous referees whose comments and suggestions have helped us in improving the paper.}




\date{\today}




\begin{abstract} This paper addresses the three concepts of \textit{ consistency, stability and convergence }  in the context of compact finite volume schemes for systems of nonlinear hyperbolic conservation laws. The treatment utilizes the framework of ``balance laws''. Such laws express the relevant physical conservation laws in the presence of discontinuities. Finite volume approximations employ this viewpoint, and the present paper can be regarded as being in this category. It is first shown that under very mild conditions a weak solution is indeed a solution to the balance law. The schemes considered here allow the computation of several quantities per mesh cell (e.g., slopes) and the notion of consistency must be extended to this framework. Then a suitable convergence theorem is established, generalizing the classical convergence theorem of Lax and Wendroff. Finally,  the limit functions are shown to be entropy solutions  by using a notion of ``Godunov compatibility'', which serves as a substitute to the entropy condition.

\end{abstract}
\maketitle

\section{\textbf{INTRODUCTION}}

       The foundational ``Lax Equivalence Theorem'' highlighted the close connection among the three concepts:  consistency, stability  and convergence. Indeed, in the context of linear evolution equations it  asserts that a consistent scheme is stable if and only if it is convergent \cite {richtmyer}. Although it was formulated for \textit{linear} evolution equations, it played a decisive role in the development of first order finite-difference or finite volume schemes, with special emphasis on hyperbolic conservation laws. The growing use of high-order schemes in this context has made it difficult to adapt the theorem in a straightforward fashion. This is particularly true for nonlinear hyperbolic conservation laws, where the
 presence (and formation) of  discontinuities does not easily allow the examination of ``consistency'' by standard Taylor expansions. Furthermore, compact (high-order) schemes require the computation of several quantities per mesh cell (e.g., slopes). As we shall see, in this case many of the existing definitions of consistency are not applicable.

  Another classical theorem in the context of approximations to conservation laws, namely, the  ``Lax-Wendroff Convergence Theorem'', uses a particular notion of consistency (see Definition ~\ref{deflwconsist} below) in order to establish convergence of discrete approximate solutions to weak solutions. Introducing a different notion of consistency necessarily forces a  revision of this convergence theorem.

 This paper addresses these issues by utilizing the framework of ``balance laws'', in one space dimension. Such laws are closely associated with the relevant physical  laws. Finite volume approximations employ this viewpoint, and the present paper can be regarded as belonging to this category.

  The aim of this paper is two-fold:
             \begin{itemize}
             \item Suggesting a version of the consistency condition that will be directly applicable to a wide array of schemes, allowing in particular for  high order, compact schemes, and taking into account the presence of discontinuities.
             \item Linking the consistency  of the discrete solution to the question of its convergence to the  solution of the balance law.
             \end{itemize}

             The focus here is on nonlinear hyperbolic balance laws, where discontinuities are formed, even for smooth initial data.

      To recall the meaning of ``consistency'', let us consider a vector function $u(x,t)$ taking values in $\RR^D$ and satisfying an evolution equation in $\RR\times\Rplusc$ of the form

      \be\label{eqevolution}
       u_t=\Phi(u),\quad t>0,
      \ee
        where $u_t=\dert u$ is the partial derivative of $u(x,t)$ with respect to the time variable $t,$ and $\Phi(u)$ is a general (not necessarily linear) operator involving spatial (with respect to $x-$variable) derivatives of $u$ (and possibly explicit dependence on $x$).

        Fixing a time interval $\Delt>0,$ the discretization procedure (or ``difference method'') is aimed at finding a sequence of functions $\set{\widetilde{u^n}(x)}_{n=0}^\infty,$ assumed to approximate the discrete sequence of values of the exact solution $\set{u(x,t_n),\,t_n=n\Delt}_{n=0}^\infty.$  The function $\widetilde{u^0}(x)$ approximates the initial data $u(x,0).$ The approximating functions are usually taken from a subspace of the ``admissible'' functions, those on which the operator $\Phi$ is acting (in some weak sense).

         Generally speaking (for a ``one step procedure''), there is a family of operators  $\set{\Phi_{\Delt},\,\Delt>0}$ generating the discrete (in time) sequence
\be
\label{eqPhikUn}\widetilde{u^{n+1}}(x)=\widetilde{u^n}(x)+\Phi_{\Delt}\widetilde{u^n}(x).
\ee
The common definition of consistency is the following \cite[Section 3.2]{richtmyer},~\cite[Section 5.4]{morton},~\cite[Section 4.1]{Tadmor}.

           \begin{defn}\label{defnconsistfirst} Let $q>0.$  The discrete scheme ~\eqref{eqPhikUn} is \textbf{consistent of order} $q$  with Equation ~\eqref{eqevolution} in the time interval $[0,T]$  if the exact solution satisfies
\be
\label{eqdefineconsist} u(x,t_{n+1})-[u(x,t_n)+\Phi_{\Delt}u(x,t_n)]=O(\Delt^{1+q}),\quad n\Delt<T.
\ee
           \end{defn}

           Observe that the equality ~\eqref{eqdefineconsist} involves a suitable norm on functions of $x,$ pertinent to the solution and approximating functions.

           \begin{rem}\label{remoldconsist1+q} In concrete cases, the verification of ~\eqref{eqdefineconsist} relies on analytical tools, notably Taylor's theorem. This poses a difficulty, since the solution is often discontinuous, as in the case of nonlinear hyperbolic conservation laws. The purpose of this paper is to introduce a consistency condition that is \textbf{meaningful also in the case of discontinuous solutions.}
           \end{rem}

             The discrete scheme ~\eqref{eqPhikUn} is actually only ``semi discrete'', since only the (continuous) time is   replaced by finite time steps. In practice, in a wide array of schemes (notably ``finite volume'')  the spatial coordinates are also discretized.  Restricting to a one-dimensional framework, a constant mesh size $\Delx>0$ is chosen, so that the ratio
         $$\
         \lambda=\frac{\Delt}{\Delx}\,\,\,\mbox{is a constant}.
         $$
   The approximating function $\widetilde{u^n}(x)$ is replaced by a discrete sequence $\widetilde{u^n_{disc}}=\set{\widetilde{u^n_j}}_{j=-\infty}^\infty,$ that is presumed to approximate the exact values
$\set{u(x_j,t_n}_{j=-\infty}^\infty$ at the spacetime grid points $\set{x_j=j\Delx,\,t_n=n\Delt}_{j=-\infty}^\infty.$

         Accordingly, the semi discrete $\Phi_{\Delt}$ is replaced by a fully discrete operator $\Phi_{\Delt,\Delx}$ and Equation ~\eqref{eqPhikUn} is replaced by a fully (i.e., spatial and temporal) discrete scheme
         \be\label
         {eqPhikUndisc}\widetilde{u^{n+1}_{disc}}=\widetilde{u^n_{disc}}+\Phi_{\Delt,\Delx}\widetilde{u^n_{disc}}.
         \ee

         It is clear how to formulate the consistency Definition ~\ref{defnconsistfirst} in this case:

  \begin{defn}
  \label{defnconsistfulldisc} Let $q>0.$ Denote by $u^n_{disc}=\set{u(x_j,t_n}_{j=-\infty}^\infty$ the set of values of $u$ at the spatial grid at time $t_n.$  The discrete scheme ~\eqref{eqPhikUndisc} is \textbf{consistent of order} $q$  with Equation ~\eqref{eqevolution} in the time interval $[0,T]$  if
           \be
           \label{eqdefineconsistdisc}u^{n+1}_{disc}-[u^n_{disc}+\Phi_{\Delt,\Delx}u^n_{disc}]=O(\Delt^{1+q}),\quad n\Delt<T.
           \ee
    \end{defn}

         In order to focus on systems of hyperbolic conservation laws of the form
 \be\label{eqconslaw} u_t+f(u)_x=0,\quad u,\,f(u)\in\RR^D,\quad (x,t)\in\RR\times\RR_+,
 \ee
 we now consider the issue of consistency of a fully discrete approximation to ~\eqref{eqconslaw}.

         The classical definition of consistency, introduced by Lax and Wendroff ~\cite{lax-wendroff}, involves a Lipschitz continuous function of $2l$ variables $g(\xi_1,\ldots,\xi_{2l})\in\RR^D,$ so that Equation ~\eqref{eqPhikUndisc} can be rewritten as
 \be\label{eqlaxwscheme}\aligned
           \widetilde{u^{n+1}_j}=\widetilde{u^{n}_j}-\lambda\Big[g(\widetilde{u^n_{j-l+1}}, \widetilde{u^n_{j-l+2}},\ldots,\widetilde{u^n_{j+l}})-g(\widetilde{u^n_{j-l}}, \widetilde{u^n_{j-l+1}},\ldots,\widetilde{u^n_{j+l-1}})\Big],\\
           -\infty<j<\infty.
                   \endaligned\ee

           \begin{defn}~\cite[Lax and Wendroff]{lax-wendroff}\label{deflwconsist} The scheme ~\eqref{eqlaxwscheme} is consistent with Equation  ~\eqref{eqconslaw} if
           \be\label{eqconsistlw}g(\xi,\ldots,\xi)=f(\xi),\quad \xi\in\RR.\ee
           \end{defn}

           This definition has proved to be very useful in the case of first-order schemes. The Lax-Wendroff theorem ~\cite{lax-wendroff} ensures that the approximate solutions obtained by a consistent and conservative scheme  ~\eqref{eqlaxwscheme}, and subject to some boundedness and (weak) convergence hypotheses, actually converge to a weak solution of ~\eqref{eqconslaw}. We refer to ~\cite{despres,elling,gallouet-herbin-latche,kroner} and references therein for various extensions of this convergence theorem (assuming first-order consistency).

           \begin{rem}\label{rem:gallouet-consist} The consistency definition ~\eqref{eqconsistlw} applies both to finite difference and finite volume schemes for \textit{uniform} grids. However, in the case of discrete approximations on non-uniform grids this definition needs to be carefully considered, since a given approximation can be consistent in the ``sense of finite differences'' but not so in the ``sense of finite volume'' schemes ~\cite[Remark 21.1]{eymard-gallouet}.
           \end{rem}

           Practical applications, as well as mathematical interests, require ``high order'' accuracy, or, using conventional  terminology, ``high order schemes''. Such a requirement can be accommodated by either one of two ways.

           \begin{itemize}
           \item Take a sufficiently large $l$ in ~\eqref{eqlaxwscheme}. In other words, extend considerably the ``stencil'' of dependence when evaluating $\widetilde{u^{n+1}_j}.$ This in turn involves a more complicated treatment of boundary conditions.
               \item Instead of considering only ``cell averages'' (where the value $\widetilde{u^n_j}$ is viewed as an average of the approximate solution  in the interval $(x_j-\frac{\Delx}{2},x_j+\frac{\Delx}{2})$), at time $t_n,$ use \textit{more information for each interval}, such as slopes or higher moments. This leads to a more complex discrete operator $\Phi_{\Delt,\Delx}$ in Equation \eqref{eqPhikUndisc},
 but enables the use of a ``compact scheme'',where only the neighboring intervals, centered at $x_{j\pm 1},$ are involved in determining the approximate solution $\widetilde{u^{n+1}_j}.$
       \end{itemize}
               The second alternative above is the one that is most widely implemented in various state-of-the-art schemes, such as MUSCL ~\cite{B.vanLeer-79}, GRP ~\cite{BenArtzi-Falcovitz-84,BenArtzi-Falcovitz-2003,MBA-Jiequan-Numer}, ADER ~\cite{Toro}, PPM ~\cite{colella}, DG ~\cite{Shu}, WENO ~\cite{Liu-Osher,Shu}. We refer also to the survey paper ~\cite[Section 3.3]{Tadmor}.
             For all these schemes,  the concept of consistency must be clearly defined and its connection to the question of convergence to the exact solution should be clarified.

             Remark that in our discussion of the system ~\eqref{eqconslaw} the flux function $f(u)$ depends only on the unknown $u(x,t).$ It is very natural (both mathematically and in applications) to try and extend this to more general flux functions. In this case, the issues of consistency and convergence should be addressed. While there is no general framework for such extensions, there are many studies of particular cases, for example ~\cite{seguin-vovelle}.

             The outline of the paper is as follows.

             Section ~\ref{secbalancelaw} deals with the basic definition of a hyperbolic ``balance law''. In order to show that the classically defined weak solutions are indeed solutions to the balance law, the continuity properties of the associated fluxes need to be studied. Theorem ~\ref{thmweakbalance} states that under very general conditions (certainly satisfied by ``entropy solution'') these fluxes are in fact locally Lipschitz continuous. It seems to be a new fact even in the case of scalar conservation laws.

             Section ~\ref{secconsistency} introduces the notion of ``approximate fluxes'' and their ``order of consistency'' (Definition ~\ref{def:consistalpha}). These notions rely on the order of the spaces used in the approximation (Assumption ~\ref{assumeVk}). In particular, it is shown (Corollary ~\ref{corconsistPhik}) that the order of consistency as introduced here conforms with the previous notion when the latter is applicable. A general definition of \textit{finite volume schemes} (FVS) is introduced (Definition ~\ref{defnapproxsoln}), that also depends on  the order of the approximating space.

             It should be emphasized that \textit{in case of discontinuous solutions, our definition of consistency may yield different orders than commonly used.} Thus, the Godunov scheme (Example ~\ref{examplegodunov1}) is of infinite order when applied to piecewise-constant functions but of order zero when applied to piecewise-linear functions. Similarly, the GRP or MUSCL schemes ~\eqref{equpgradeGodflux} are only first-order consistent, while second-order consistency requires  smooth solutions.

              Section ~\ref{seclaxwendroff} deals with the convergence of the approximate solutions to the exact solution of the balance law. It is interesting to recall here the following paragraph from DiPerna's paper ~\cite{diperna}: ``In the setting of the scalar conservation it remains an open problem to establish convergence of conservative finite difference schemes which are accurate to second order. Stability and convergence results have been obtained so far only for methods which are precisely accurate to first order.''

              Since then, the convergence of various second-order schemes to the unique entropy solution (in the scalar case) has been established in ~\cite{DSCS, Bouchut-Bou-Perthame, chain, goodman,LeFloch-Liu-99, Lions-Souganidis-95,S.Osher-85, Vila}. These studies treated specific schemes and employed suitable discrete entropy inequalities. In particular, they had no need to introduce a general concept of consistency as had been done in the first-order case described above.

            In Theorem ~\ref{thmconverge} the convergence of the approximate solutions to solutions of the balance law is proved for a general finite volume scheme, under certain boundedness conditions. The hypothesis that the scheme is consistent of order $q>0$ plays a crucial role. The natural question to be asked is whether or not the limit functions satisfy the entropy condition. In order to provide an affirmative answer the concept of ``Godunov compatibility'' is introduced in Definition  ~\ref{defineGodtype} and is used in Theorem ~\ref{thmunique}. It should be noted that the compatibility condition is based on the assumption that the Godunov scheme converges to an entropy solution; this fact has actually been proved  in the scalar case ~\cite{GodlewskiRaviart} where it is known to be unique and in a class of $2\times 2$ systems by DiPerna ~\cite{diperna}. We refer also to ~\cite{ding-chen} in the case of isentropic gas dynamics.

\section{\textbf{THE FUNDAMENTAL PRINCIPLE OF THE HYPERBOLIC BALANCE LAW}}\label{secbalancelaw}
\subsection{\textbf{GENERAL SYSTEMS IN ONE SPACE DIMENSION }}\label{subsecsystemlaw}

    We now focus on the case of a system of conservation laws ~\eqref{eqconslaw}
    $$
    u_t+f(u)_x=0,\quad x\in \RR,\,\,t\geq 0,\,\,\,u,f(u)\in\RR^D,
    $$
    subject to initial data
    \be\label{eqinitdata}
    u(x,0)=u_0(x),\quad x\in\RR.
    \ee

     We consider the case of a single space dimension. In order to avoid complications caused by the presence of boundaries, we limit our considerations to the pure initial value problem set on the whole line $\RR.$

 Formally, by integration we infer that for  every rectangle $Q=[x_1,x_2]\times[t_1,t_2]\subseteq \RR\times\overline{\RR_+}$ the following equality holds.

    \be\label{eqbalancecons}
     \int_{x_1}^{x_2}u(x,t_2)dx-\int_{x_1}^{x_2}u(x,t_1)dx=
     -\Big[\int_{t_1}^{t_2}f(u(x_2,t))dt-\int_{t_1}^{t_2}f(u(x_1,t))dt\Big].
    \ee

\begin{rem}\label{remgaussgreen} Equation ~\eqref{eqbalancecons} can be considered as an integrated (formal) form of ~\eqref{eqconslaw}, using the Gauss-Green theorem. However, the application of this theorem is certainly not straightforward, since the function $u(x,t)$ is not even continuous (see ~\cite[Section 4.5]{federer}). We refer to ~\cite{chen} and ~\cite[Chapter I]{dafermos} for an abstract discussion of this topic. Regarding the right-hand side of ~\eqref{eqbalancecons} one needs to keep in mind the following comment concerning the identification of the boundary flux:``the drawback of this, functional analytic, demonstration is that it does not provide any clues on how the $q_\mathfrak{D}$ may be computed from $A$'' ~\cite[Section 1.3]{dafermos}.
    \end{rem}

    In fact, the meaning of the $x$ and $t$ derivatives must be clarified since the solutions generate discontinuities, such as shocks or interfaces. As is well known, the concept of a \textbf{weak solution} is introduced  precisely in order to handle this difficulty  ~\cite[Chapter 11]{evans}, as follows.

    For every rectangle $Q=[x_1,x_2]\times[t_1,t_2]\subseteq \RR\times\Rplusc,$ if $\phi(x,t)\in C^\infty_0(Q),\,$ then
        \be\label{eqweaksol}
        \int_{t_1}^{t_2}\int_{x_1}^{x_2}[u(x,t)\phi_t+f(u(x,t))\phi_x]dx\,dt=0.
        \ee

        In light of the above comments, the proof of the following theorem is not so obvious.

      \begin{thm}\label{thmweakbalance} Let $u(x,t)$ be a weak solution to the system ~\eqref{eqconslaw}, with initial function $u_0\in L^1(\RR)\cap L^\infty(\RR).$

      Assume that $u(x,t)$ satisfies the following properties.

      \begin{itemize}\item $u(x,t)$ is locally bounded in $\RR\times\Rplusc.$

      \item For every fixed interval $[x_1,x_2]\subseteq\RR$ the mass \be\label{eqmtcont} m(t)=\int\limits_{x_1}^{x_2} u(x,t)dx\,\,   \mbox{is a well-defined and continuous function of}\,\, t\in\Rplusc.\ee
      \end{itemize}
      Then we have:
      \begin{enumerate}
 \item[(i)] For every fixed $[t_1,t_2]\subseteq\RR$ the integral $g(x)=\int_{t_1}^{t_2}f(u(x,t))dt$ is locally
 Lipschitz continuous in $x\in\RR.$
\item[(ii)]  $u(x,t)$ satisfies the equality ~\eqref{eqbalancecons} in every rectangle $Q$.
\end{enumerate}
      \end{thm}

   \begin{proof}For every rectangle $Q=[x_1,x_2]\times[t_1,t_2]\subseteq \RR\times\Rplusc$ we define
   \be\label{equxtmaxu0}
      C_Q=\sup\set{ |u(x,t)|, \quad (x,t)\in Q}.
       \ee
    Note that in ~\eqref{eqbalancecons},  the ``fixed time'' integrals in the left-hand side exist by the assumed continuity (in time) of $m(t).$ Pick $\phi(x,t)=\theta(t)\psi(x)$ in Equation ~\eqref{eqweaksol}, where $\theta\in C^\infty_0(t_1,t_2)$ and $\psi\in C^\infty_0(x_1,x_2).$ Take $0\leq\theta\leq 1$ and $\theta(t)=1$ for $t_1+\eps\leq t\leq t_2-\eps.$ Letting $\eps\to 0,$ Equation ~\eqref{eqweaksol} yields
       \be\label{eqintuf}
        \int_{x_1}^{x_2}[u(x,t_2)-u(x,t_1)]\psi(x)dx=\int_{x_1}^{x_2}\big\{\int_{t_1}^{t_2}f(u(x,t))dt\big\}\psi'(x)dx.
       \ee
Letting  $g(x)=\int_{t_1}^{t_2}f(u(x,t))dt,$ Equation ~\eqref{eqintuf} can be rewritten as
 $$
          \int_{x_1}^{x_2}[u(x,t_2)-u(x,t_1)]\psi(x)dx=\int_{x_1}^{x_2}g(x)\psi'(x)dx.
$$
Since $|u(x,t)|\leq C_Q$ it follows that
         $$
         \Big|\int_{x_1}^{x_2}g(x)\psi'(x)dx\Big|\leq 2C_Q\|\psi\|_1.
         $$
 Define the linear functional
          $$
          \Gcal\psi=\int_{x_1}^{x_2}g(x)\psi'(x)dx,\quad \psi\in C^\infty_0(x_1,x_2).
          $$
          The above estimate shows that $\Gcal$ is continuous with respect to the $L^1$ norm. The density of $C^\infty_0(x_1,x_2)$ in $L^1(x_1,x_2)$ and the $L^1,\,L^\infty$ duality entail that there exists a function $r(x)\in L^\infty(x_1,x_2)$ such that
          \be\label{eqdualgr}
           \int_{x_1}^{x_2}g(x)\psi'(x)dx=\int_{x_1}^{x_2}r(x)\psi(x)dx,\quad \psi\in C^\infty_0(x_1,x_2).
          \ee
           Since $r\in L^p(x_1,x_2)$ for all $p<\infty$ it follows that $g\in W^{1,p}(x_1,x_2),$  the Sobolev space of order $p$ for all $p<\infty.$ Turning back to ~\eqref{eqdualgr} we see that the distributional derivative of $g(x)$ satisfies $g'(x)=-r(x)$ in $(x_1,x_2).$ Since the above estimates depend only on $Q,$ the function $g(x)$ is Lipschitz  in $\RR.$

             An explicit estimate of the Lipschitz constant is readily obtained from Equation ~\eqref{eqintuf}: the above argument (for the $t$ variable) can be repeated; take $0\leq\psi\leq 1$ and $\psi(x)=1$ for $x_1+\eps\leq x\leq x_2-\eps.$ Letting $\eps\to 0$  yields
       $$
       \int_{x_1}^{x_2}[u(x,t_2)-u(x,t_1)]dx=-[g(x_2)-g(x_1)].
       $$
    This establishes the validity of the equality ~\eqref{eqbalancecons} and moreover
          $$
          |g(x_2)-g(x_1)|\leq 2C_Q |x_2-x_1|.
          $$
       \end{proof}
           \begin{rem}\label{remmtcont} We could replace the continuity
           assumption ~\eqref{eqmtcont} by the stronger assumption that the map $t\to u(\cdot,t)\in L^\infty(\RR)\,\mbox{weak}^\ast$ is continuous. This latter assumption is universally imposed when dealing with entropy solutions to nonlinear conservation laws ~\cite[Section 4.5]{dafermos}. However the continuity condition ~\eqref{eqmtcont} is valid for weak solutions that are not necessarily entropy solutions. In fact, it holds for weak solutions that have bounded (locally in time) total variation. This is expressed by Dafermos as ``mechanism of regularity transfer from the spatial to the temporal variables'' ~\cite[Theorem 4.3.1]{dafermos}.
           \end{rem}

       The statement of Theorem ~\ref{thmweakbalance} is closely related to the more fluid dynamical viewpoint: the ``conservation law'', which is a \textbf{partial differential equation,} is replaced by a \underline{``balance law''}. This latter viewpoint plays a central role in this paper, and is introduced in the following paragraphs.

       We assume the existence of two Banach spaces, $\fU^b$ and $\fU^c,$ with respective norms $\|\cdot\|_b,\,\,\|\cdot\|_c,$ and set \be\label{eqdeffU}\fU=\fU^b\cap\fU^c.\ee

       \begin{defn}\label{defnnormU}
           The norm in $\fU$ is
       \be\label{eqnormU}
      \|w\|_{\fU}=\|w\|_b+\|w\|_c.
      \ee
     \end{defn}
$\fU$ is assumed to be a ``persistence space'' for the class of solutions introduced below,  in the sense that they satisfy, for every initial data $u_0\in \fU,$
   \be\label{eqpersist}
   t\hookrightarrow u(\cdot,t)\in C(\Rplus,\fU^c)\cap L^\infty(\Rplusc,\fU^b).
   \ee
      Note that as in the case of weak solutions, no uniqueness assumption is imposed at this stage.

 We shall also make use of the Fr\'{e}chet space $L^1_{loc}(\RR)$ whose  metric is given (for two vector-valued functions $f,\,g$) by
       \be\label{eqL1locmetric}
      d(f,g)=\suml_{N=1}^\infty 2^{-N}\,\,\frac{\int\limits_{-N}^N |f(x)-g(x)|dx}{1+\int\limits_{-N}^N |f(x)-g(x)|dx},\quad f,g\in L^1_{loc}(\RR).
      \ee

    \begin{defn}\label{defnbalance}  Let $u_0\in\fU.$ The function $u(\cdot,t)\in C(\Rplusc,\fU^c)\cap L^\infty(\Rplus,\fU^b)$ is a solution to the balance law ~\eqref{eqbalancecons} corresponding to the partial differential equation  ~\eqref{eqconslaw}  if the following  conditions are satisfied.
    \begin{itemize}
    \item  For every $x\in\RR$ and interval  $[t_1,t_2]\subseteq \overline{\RR_+}$ the integral $\int_{t_1}^{t_2}f(u(x,t))dt$ is well defined, and is a continuous function of $x\in\RR.$
    \item  For every $t\geq 0$ and interval  $[x_1,x_2]\subseteq \RR$ the integral $\int_{x_1}^{x_2}u(x,t)dx$ is well defined and is a continuous function of $t.$
     \item For every rectangle $[x_1,x_2]\times[t_1,t_2]\subseteq \RR\times\Rplusc$ the balance equation ~\eqref{eqbalancecons} is satisfied.





    \end{itemize}
    \end{defn}
    \begin{rem}\label{rembalancelaw} Our definition of a solution to the balance law conforms to that introduced in ~\cite[Chapter I]{dafermos}. In fact, in Dafermos' book the balance equation is assumed to hold for any domain in spacetime. We note that other authors use various other terms, such as the ``integral conservation law'', and the term ``balance law'' is applied to a conservation law with a source term.
    \end{rem}

      Definition ~\ref{defnbalance} is closely related to the physical interpretation of systems of conservation laws, in particular the Euler system of compressible fluid flow. Furthermore, the balance law serves as the foundation of numerical finite volume schemes; in fact, every interval $[x_1,x_2]$   is considered as a ``control volume'' in which the balance law is satisfied between arbitrary time levels $t_1<t_2.$

      Theorem ~\ref{thmweakbalance} implies that a weak solution satisfying certain hypotheses (in particular an entropy solution) is a solution to the balance law  in the sense of Definition ~\ref{defnbalance}. It is easy to see that conversely, a solution to the balance law  is a weak solution of the  conservation law ~\eqref{eqconslaw}.

       For notational simplicity we shall occasionally denote by $S(t)$ the  solution operator,
        \be\label{eqdenoteSt}
        u(\cdot,t)=S(t)u_0(\cdot),
        \ee
     even though the solution is not assumed to be unique.

\subsection{\textbf{THE SCALAR CONSERVATION LAW}}\label{subsecscalarlaw}
 We now confine the above discussion to the scalar equation, namely $u\in\RR.$
We  refer to the classical paper ~\cite{kru} and to the books ~\cite{dafermos,evans,GodlewskiRaviart} for the notion of the Kru\v{z}kov \textbf{entropy solution.} Note that in Theorem ~\ref{thmweakbalance} we do not need to assume that $u(x,t)$ is an entropy solution.

In  this case  the function space $\fU^b$ (see ~\eqref{eqdeffU}) is taken as the space $L^\infty(\RR)$ while $\fU^c$ is the space $L^1(\RR)$.

    The norm in $L^p(\RR)$ is denoted by $\|w\|_p.$

      We recall the basic facts concerning this evolution semigroup ~\cite[Chapter 2]{GodlewskiRaviart}:

      \begin{claim}\label{claimStproperty} The solution semigroup $S(t):\fU\hookrightarrow C(\overline{\RR_+},L^1(\RR))$ is continuous and satisfies
      \begin{itemize}
      \item[(i)] $$
      \|S(t)u_0\|_\infty\leq \|u_0\|_\infty, \quad t\geq 0.
      $$
      \item[(ii)]
         $$ \|S(t)u_0-S(t)v_0\|_1\leq \|u_0-v_0\|_1, \quad t\geq 0.
         $$
      \end{itemize}
     \end{claim}

      Remark in particular that for any fixed interval $[x_1,x_2]\subseteq\RR$ the mass $m(t)=\int\limits_{x_1}^{x_2} u(x,t)dx$ of the entropy solution $u(x,t)$  is well-defined and, indeed, is a continuous function of $t\in\Rplusc.$

%


%








\section{\textbf{CONSISTENCY OF APPROXIMATE FLUXES}}\label{secconsistency}

 Consider the balance law for systems ~\eqref{eqbalancecons}. In this section we introduce approximate fluxes associated with it. These fluxes serve in the construction of \textbf{compact schemes,}  designed to approximate the solution of the balance law.

\fbox{\parbox[c]{300pt}{Let $k=\Delta t>0.$   As is common in the literature on finite difference methods (for evolution equations) we set $k$ as the \underline{sole} parameter in the study. Thus, convergence of approximate solutions to the exact ones will be studied in terms of limits as $k\to 0.$}}

The spatial step is $h=\Delta x=\lambda^{-1}k,$ where $\lambda>0$ is assumed to be fixed.

\begin{defn}

    The $k-$ spatial grid is the discrete set in $\RR,$
    $$
    \Gamma_k=\set{x_j=jh}_{j=-\infty}^\infty,
    $$
    and the \textbf{grid intervals} (or  \textbf{grid cells}) are the intervals
    $$
    I_j=\Big(x_{j-\frac12},x_{j+\frac12}\Big),\quad -\infty<j<\infty,
    $$
    where
    $$
    x_{j\pm\frac12}=x_j\pm\frac{h}{2}.
    $$

The \textbf{spacetime $k-$ grid} is the discrete set
 \be\label{eqkgrid}
 \GRST=\Gamma_k\times\set{t_n=nk,\,\,n=0,1,2,\ldots}.
 \ee
 \end{defn}

Recall the persistence space $\fU$ as in Equation ~\eqref{eqpersist}. Given a spatial grid $\Gamma_k$ we assume that:
\begin{assume}\label{assumeVk}
 \begin{itemize}
 \item[(i)] There exists a \textbf{functional subspace} $V^k\subseteq \fU$ having the following property:

 The restrictions of the elements of $V^k$  to any grid interval $I_j$ constitute a finite dimensional subspace.
  The dimension of these restrictions is called the \textbf{order of $V^k.$}
 Typically, it is a space of piecewise polynomial functions of fixed degree, with possible discontinuities at the boundary points $\set{x_{j\pm\frac12}}_{j=-\infty}^{j=\infty}$ of every grid interval $I_j.$

     \item[(ii)] There exists a projection $P^k:\fU\to V^k,$
       that does not change averages in grid cells, namely,
    \be\label{eqinvariantave}
    \int_{I_j}P^kv(x)dx=\int_{I_j}v(x)dx,\quad -\infty<j<\infty,\,\,v\in\fU.
         \ee
  \end{itemize}
  \end{assume}

  \begin{rem}\label{remnotationk}
  \begin{enumerate}
   \item[(i)]  The notation of the space $V^k$ (and the projection $P^k$)  refers explicitly only to the variable parameter $k$ (that determines  the spatial step $h=\lambda^{-1}k$). However, this space also depends on our choice of the dimension of its restrictions to grid intervals, such as piecewise-constant (``first order''), piecewise-linear (``second order'') and so on.
   \item[(ii)]

   The operators $P^k$ are sometimes called ``reconstruction operators''. They involve suitable interpolations
   and  ``slope limiters''.
  \end{enumerate}
  \end{rem}

     \textbf{Notation.} Elements of $V^k$ will be designated by Greek letters:  $\xi\in V^k.$
         There will be \underline{no other use} of Greek letters throughout the paper (except for the fixed constant ratio $\lambda=\frac{k}{h}$).

         \subsection{\textbf{APPROXIMATE FLUXES}}
      The  approximate fluxes introduced here are intended to be sufficiently general, so as to cover a wide variety of  finite volume schemes of any order.

 We assume that there exists $\lambda_0>0$ so that for every fixed $\lambda=\frac{ k}{h}<\lambda_0$  and every $0<k<\frac12 T,\,h=\lambda^{-1}k,$ there exists
   \be\label{eqnumerflux}
   \aligned \mbox{  a  sequence of continuous functions} \,\,\set{F^{\xi}_{j+\frac12}(t),\,\,0\leq t<k}_{j=-\infty}^\infty,\\\mbox{for every}\,\, \,\,\xi\in V=V^k.\,\,\hspace{180pt} \\\endaligned
 \ee

  \begin{defn}[\textbf{Approximate Fluxes}]\label{defnnumericflux}

   We say that the functions of the family $\set{F^{\xi}_{j+\frac12}(t),\,\,0\leq t<k}_{j=-\infty}^\infty$ are \textbf{approximate fluxes} (in the time interval $[0,k)$) corresponding to the initial function $\xi\in V^k,$ if the following \textbf{finite propagation property} is satisfied.

  $F^{\xi}_{j+\frac12}(t),\,\,0\leq t<k,$ depends only on the restriction of \, $\xi$ to $I_j\cup I_{j+1}.$

            Furthermore, if \, $\xi\equiv c=const.$ in $I_j\cup I_{j+1}$ then $F^{\xi}_{j+\frac12}(t)\equiv f(c).$
         \end{defn}

         Note that in this definition the grid points $\set{x_{j+\frac12}}_{j=-\infty}^\infty$ satisfy $x_{j+\frac12}-x_{j-\frac12}=h=\lambda^{-1}k.$

         \begin{rem}\label{remCFL}
         The assumption above that $\lambda$ is sufficiently small is the ``CFL condition'' that enables the finite propagation property of the fluxes.
         \end{rem}

              The terminology of ``approximate fluxes''  is suggested by the fact that they are viewed as approximating the flux values $f(u)$ at the nodes $\set{x_{j+\frac12}}_{j=-\infty}^\infty$ in a sense that will be made rigorous below ~\eqref{eqdefnconsist}.
        \begin{rem}[\textbf{Uniformity of the spatial grid}]\label{remunifgrid} While in our treatment the time step $k>0$ is constant over the whole mesh, the spatial grid may be \textit{non uniform}. This is due to the fact that Definition ~\ref{defnnumericflux} relies only on fluxes restricted to cell boundaries. We have chosen to avoid this generality since it leads to notational complications (for example, the underlying discrete spaces $V^k$ consist of piecewise polynomial functions over cells of variable size).
        \end{rem}
         \subsection{\textbf{CONSISTENCY}}
         We now proceed to define the key concept of \textit{consistency.} As explained in the Introduction, the idea of ``\textit{consistency}'' involves a comparison between exact and approximate solutions, over short time intervals. In preparation we need to introduce suitable families of initial data, contained in the spaces $V^k.$ For such initial data we assume the short-time existence of unique solutions to the balance law, as in Assumption ~\ref{assumesoln} below.  We take  $\lambda_0>0$ as in Definition ~\ref{defnnumericflux} and consider spacetime grids satisfying $\lambda=\frac{ k}{h}<\lambda_0.$

         \begin{defn}\label{def:initialcompact} Let $H \subseteq \bigcup\limits_{0<k<\frac12 T}V^k.$ We say that $H$ is an \textbf{admissible set of initial data} if for every $\lambda=\frac{ k}{h}<\lambda_0$ the set $H\cap V^k$ is bounded in the $\fU$  topology and compact in the $L^1_{loc}(\RR)$ topology ~\eqref{eqL1locmetric}.
         \end{defn}

           As an example, we can think of $H$ (in the scalar case) as the set of uniformly  bounded functions having a finite \textit{total variation.}

           \begin{assume}\label{assumesoln}
            Let $\xi(x)\in V^k.$ Then the balance law ~\eqref{eqbalancecons} admits a unique solution in the time interval $t\in [0,k],$  subject to the initial condition $\xi(x).$ The uniqueness is achieved  by imposing suitable constraints, such as ``entropy conditions.'' This solution is denoted henceforth by $u(x,t;\xi)=S(t)\,\xi\in\fU,\,t\in [ 0,k].$
          \end{assume}

           Actually, under some additional boundedness hypotheses, Assumption ~\ref{assumesoln} can be verified ~\cite{Li-Yu}:
           \begin{claim}\label{claimexistentropy} Let $V^k\subseteq \fU$ be of any (finite) order. Then for every $\xi\in V^k$ there exists a unique entropy solution $u(\cdot,t;\xi)=S(t)\,\xi(\cdot),\,\,t\in [ 0,k].$ This solution can be obtained by a constructive procedure, using characteristic curves and generalized Riemann solvers.
           \end{claim}

         \begin{defn}\label{def:consistalpha} Consider the setup as in Definition ~\ref{defnnumericflux}. Let $q\geq 0.$ The approximate fluxes $\set{F^{\xi}_{j+\frac12}(t),\,\,0\leq t<k}_{j=-\infty}^\infty$ are said to be \textbf{consistent of order $q$} with the balance law ~\eqref{eqbalancecons} if  for every admissible set of initial data $H$ \,\,and all \,\,$\xi\in H\cap V^k,$
   \be\label{eqdefnconsist}
   \aligned\int_{0}^{k}\Big[F^{\xi}_{j+\frac12}(t)-F^{\xi}_{j-\frac12}(t)\Big]dt-
   \int_{0}^{k}\Big[f(u(x_{j+\frac12},t;\xi))-f(u(x_{j-\frac12},t;\xi))\Big]dt\\
   \leq Ck^{2+q},\quad -\infty<j<\infty,\hspace{200 pt}\endaligned
   \ee
   where $C>0$ depends only on $H.$


   \end{defn}

   \begin{rem}\label{remorder2q}
   Observe that the order of consistency in Definition ~\ref{def:consistalpha} depends on the \textit{choice of the space $V^k.$ } This will be illustrated in Example ~\ref{examplegodunov1} below.

   Also, the exponent $2+q$ is related to the exponent $1+q$ in ~\eqref{eqdefineconsistdisc}. This will be further discussed in Corollary ~\ref{corconsistPhik} below. As already noted in Remark ~\ref{remoldconsist1+q} the order of consistency may depend on the regularity of the solution. Refer to Subsubsection ~\ref{subsubGRP} below for a detailed analysis of the interplay between regularity and order of consistency.
   \end{rem}

   \begin{rem}\label{remCorderalpha} Note that the right-hand side in ~\eqref{eqdefnconsist} is assumed to be bounded by $Ck^{2+q},$ where $C>0$ depends on $H$ but is independent of $j.$  Typically this dependence is expressed in terms of norms of the restrictions of $\xi$ to $I_j$ and neighboring grid intervals.

   Of course this can be relaxed by assuming, for instance, that the constant $C>0$ is ``localized'', so that $C=C(A),$ for all $j$ such that $ x_{j\pm\frac12}\in [-A,A].$
   \end{rem}

         \begin{example}[\textbf{The Godunov Approximate Flux}~\cite{Godunov-59}] \label{examplegodunov1} Let $V^k$ be of first order , namely, the space of piecewise constant (in grid intervals) functions. Then (if $\lambda_0$ is sufficiently small by the CFL condition) by definition
       \be\label{eqgodunov1}
       F^{\xi}_{j+\frac12}(t)=f(u(x_{j+\frac12},t;\xi)),\quad 0\leq t< k,\,\,-\infty<j<\infty,
       \ee
       so that, for this space, the approximate flux is consistent to any order.

       Recall that in this case $u(x_{j+\frac12},t;\xi)\equiv const$ is the solution to the \textbf{Riemann problem} subject to the two sided initial data $\xi_j,\xi_{j+1}.$

    \end{example}

    \subsubsection{\textbf{ ORDER OF CONSISTENCY AND REGULARITY}}\label{subsubGRP}
     Suppose now that we try to implement the Godunov approximation for the case of second-order spaces (namely, $V^k$ consists of functions that are linear in grid cells).
The approximate flux is therefore

      \be\label{eqGodunovpiecelinear}
      F^{\xi}_{j+\frac12}(t)=f(u(x_{j+\frac12},0+;\xi)),
      \ee
         where $u(x_{j+\frac12},0+;\xi)$ is the ``instantaneous'' solution to the Riemann problem subject to the two sided initial data $\xi_{j+\frac12-},\xi_{j+\frac12+},\,\,$ the limiting values of the piecewise linear function $\xi(x)$ at $x_{j+\frac12}.$

         Let $u_t(x_{j+\frac12},0+;\xi)$ be the instantaneous value of the time-derivative of the solution (this is actually the solution to the \textbf{Generalized Riemann Problem (GRP)}~\cite{BenArtzi-Falcovitz-2003,MBA-Jiequan-Numer}). From
         \be\label{eqfusecondorder}
         \begin{array}{rl}
       & f(u(x_{j+\frac12},t;\xi))\\[3mm]
        =&  f(u(x_{j+\frac12},0+;\xi))+f'(u(x_{j+\frac12},0+;\xi))u_t(x_{j+\frac12},0+;\xi)t
         +\mathcal{O}(t^2),
         \end{array}
         \ee
 it follows that
         $$
         \begin{array}{rl}
         &\displaystyle \int_{0}^{k}F^{\xi}_{j+\frac12}(t)dt-
   \int_0^k f(u(x_{j+\frac12},t;\xi))dt\\[3mm]
   =&\displaystyle \frac12 f'(u(x_{j+\frac12},0+;\xi))u_t(x_{j+\frac12},0+;\xi)k^2+\mathcal{O}(k^3).
   \end{array}
   $$
  Hence the left-hand side of  ~\eqref{eqdefnconsist} is
   \be\label{eqdiffGRP}\aligned \int_{0}^{k}\Big[F^{\xi}_{j+\frac12}(t)-F^{\xi}_{j-\frac12}(t)\Big]dt-\int_0^k\Big[ f(u(x_{j+\frac12},t;\xi))-f(u(x_{j+\frac12},t;\xi))\Big]dt\\=\frac12 [f'(u(x_{j+\frac12},0+;\xi))u_t(x_{j+\frac12},0+;\xi)-f'(u(x_{j-\frac12},0+;\xi))u_t(x_{j-\frac12},0+;\xi)]k^2\\+\mathcal{O}(k^3).\endaligned
   \ee

         If no regularity of the solution $u(x,t;\xi)$ is assumed (in particular, if it is discontinuous) then the approximate flux is only  consistent of order zero ($q=0$ in ~\eqref{eqdefnconsist}). However, in regions where the solution is smooth the difference
         \be\label{eqdiffaddq} f'(u(x_{j+\frac12},0+;\xi))u_t(x_{j+\frac12},0+;\xi)-f'(u(x_{j-\frac12},0+;\xi))u_t(x_{j-\frac12},0+;\xi)=\mathcal{O}(k),
         \ee
         thus raising the order of consistency to $q=1.$

         In view of ~\eqref{eqfusecondorder} the remedy here is to upgrade the approximate flux ~\eqref{eqGodunovpiecelinear} by adding the GRP solution, thus introducing the GRP fluxes.

  \begin{defn}[{\textbf{GRP Approximate Flux}}]\label{defGRPflux} The GRP approximate flux is given by
        \be\label{equpgradeGodflux}
        F^{\xi}_{j+\frac12}(t)=f(u(x_{j+\frac12},0+;\xi))+f'(u(x_{j+\frac12},0+;\xi))u_t(x_{j+\frac12},0+;\xi)t.
        \ee
        \end{defn}

        Now
         $$
         \int_{0}^{k}F^{\xi}_{j+\frac12}(t)dt-\int_0^k f(u(x_{j+\frac12},t;\xi))dt=\mathcal{O}(k^3),
         $$
         so that the order of consistency is $q=1$ in all cases. For smooth solutions we obtain second-order consistency ($q=2$), since in analogy with ~\eqref{equpgradeGodflux}
         \be\label{eqdiffaddqa} f''(u(x_{j+\frac12},0+;\xi))u_t(x_{j+\frac12},0+;\xi)-f''(u(x_{j-\frac12},0+;\xi))u_t(x_{j-\frac12},0+;\xi)=\mathcal{O}(k).
                \ee
                  Thus, when reduced to the smooth setting, the common statement about the second order  consistency of this approximate flux (as well as the MUSCL flux below) is recovered.

         \begin{example}[{\textbf{MUSCL  Approximate Flux}}]\label{exampleMUSCL} The derivative $f'(u(x_{j+\frac12},0+;\xi))$ in ~\eqref{equpgradeGodflux} depends solely on the Riemann solution. On the other hand, the instantaneous time derivative $u_t(x_{j+\frac12},0+;\xi)$ is obtained from the GRP solution. Suppose that we can somehow find an approximation $ v(x_{j+\frac12},0+;\xi)$ so that
       \be\label{equtvdifference}
       u_t(x_{j+\frac12},0+;\xi)-v(x_{j+\frac12},0+;\xi)=\mathcal{O}(k^\beta),\quad \beta\geq 0.
       \ee
              Let us define new approximate fluxes by
              $$ F^{\xi}_{j+\frac12}(t)=f(u(x_{j+\frac12},0+;\xi))+f'(u(x_{j+\frac12},0+;\xi))v(x_{j+\frac12},0+;\xi)t,
              $$
              so that now
              $$
              \int_{0}^{k}F^{\xi}_{j+\frac12}(t)dt-\int_0^k f(u(x_{j+\frac12},t;\xi))dt=\mathcal{O}(k^3)+\mathcal{O}(k^{2+\beta}).
              $$
              The MUSCL scheme of van-Leer ~\cite{B.vanLeer-79} provides such an approximation with $\beta=1$ ~\cite[Appendix D]{BenArtzi-Falcovitz-2003} and we conclude that it is consistent of order $q=1.$
         \end{example}

            \begin{example}[{\textbf{Acoustic GRP  Approximate Flux}}]\label{exampleacoustic}
               The acoustic GRP flux was introduced in  ~\cite[Proposition 5.9]{BenArtzi-Falcovitz-2003} and serves as a particularly simple extension of the Godunov flux. In fact, it also serves as the foundation of the ADER methodology ~\cite[p.807]{dumbser}.  It is only applicable if no strong discontinuities are present and in this case it has the same order of consistency as the MUSCL flux ~\cite[Theorem 5.36]{BenArtzi-Falcovitz-2003}, namely, $\beta=1$ in ~\eqref{equtvdifference}. However, in the presence of strong discontinuities it is consistent of order $q=0,$ hence does not offer a formal improvement of the Godunov flux. It should be noted that in simulations of problems that do not involve strong discontinuities it actually yields much better approximations than those provided by the Godunov scheme ~\cite{Toro}.
            \end{example}

             It is now clear how to obtain still higher order of consistency ($q=2$ in discontinuous cases): a second-order time derivative is added to the generalized Riemann solution. This has already been implemented in
             the case of the Euler compressible flow ~\cite{qian-li}.

          \subsection{\textbf{CONSISTENCY--COMPARING OLD AND NEW}}
          The approximate fluxes introduced above lead to the construction of approximate solutions by \textit{finite volume schemes.} This construction is introduced here, along with the order of consistency of the ensuing scheme. The compatibility of the new definition of order of consistency with the classical definition (Definition ~\ref{defnconsistfirst}) is established.

           In order to conform with the conventional treatment, we consider the general step of the scheme (namely, from $t_n$ to $t_{n+1}$).

           Assumption ~\ref{assumesoln} is imposed (see also Claim~\ref{claimexistentropy}), guaranteeing the existence of a unique solution $u(x,t;\xi),\,\,\xi\in V^k,$ to the balance law,  in every time step.

    In the following definition we assume that  $\set{F^{\xi}_{j+\frac12}(t),\,\,0\leq t<k}_{j=-\infty}^\infty$ are approximate fluxes consistent with the balance law ~\eqref{eqbalancecons} .

      \begin{defn}\label{defnapproxsoln}\begin{enumerate}
        \item[(i)] Suppose that there is a  map $\widetilde{S(k)}:V^k\to \fU$ so that, for every $\xi\in V^k,$
   \be\label{eqapproxsoln}
   \int_{I_j}\widetilde{S(k)}\xi dx-\int_{I_j}\xi dx=
   -\int_{0}^{k}\Big[F^{\xi}_{j+\frac12}(t)-F^{\xi}_{j-\frac12}(t)\Big]dt,\quad -\infty<j<\infty.
   \ee
     Then \,$\widetilde{S(k)}$ is called \textbf{an approximate evolution operator} to the balance law associated with these fluxes.
   \item[(ii)]  Let  $\set{F^{\xi}_{j+\frac12}(t),\,\,0\leq t<k}_{j=-\infty}^\infty$ be approximate fluxes consistent with the balance law ~\eqref{eqbalancecons}. We say that a family of maps
       $\set{\Phi^k:V^k\to V^k}_{k>0}$ is a \textbf{Finite Volume Scheme (FVS)}
        for the balance law ~\eqref{eqbalancecons} if
          \be\label{eqdecompPhik}
          \Phi^k=P^k\widetilde{S(k)},\ee
           where  $P^k$ is the projection as in ~\eqref{eqinvariantave}.
          \end{enumerate}
          \end{defn}

          Given a time step $k>0,$ we define a sequence $\set{\theta^n}_{n=0}^\infty\subseteq V^k$ as follows.

          First, $\theta^0=P^ku_0(x).$ We construct this sequence successively  by letting first
            \be\label{eqdefnthetan}
            u(x,t-t_n;\theta^n)=S(t-t_n)\theta^n,
            \ee
              and then
          \be\label{eqsuccexact}
          \theta^{n+1}=P^k(u(x,t_{n+1}-t_n;\theta^n)).
          \ee
            The set of cell averages of these functions is defined by
  $$
  \set{\theta^{n+1}_{j}=h^{-1}\int_{I_j}\theta^{n+1}(x)dx}_{j=-\infty}^\infty,\quad n=0,1,2,\ldots
  $$


       \begin{prop}\label{propexact} Assume that the approximate fluxes $\set{F^{\xi}_{j+\frac12}(t),\,\,0\leq t <k}_{j=-\infty}^\infty$ are consistent of order $q,$  in the sense of Definition ~\ref{def:consistalpha}.
     Then the sequence of cell averages over the intervals $I_j$ satisfies
        \be\label{eqtrunc}\aligned
           \theta^{n+1}_{j}  -\theta^{n}_{j}= -\frac{\lambda}{k}\int_{t_n}^{t_{n+1}}
           [F^{\theta^n}_{j+\frac12}(t-t_n)-F^{\theta^n}_{j-\frac12}(t-t_n)]dt+\mathcal{O}(k^{1+q})  ,\\ -\infty<j<\infty.\hspace{30pt}
       \endaligned
       \ee
       \end{prop}

       \begin{proof}
    In view of ~\eqref{eqbalancecons}
   \be\aligned
   \int_{I_j}u(x,t_{n+1}-t_n;\theta^n)dx-\int_{I_j}\theta^n(x)dx\hspace{50pt}\\=
 -\int_{t_n}^{t_{n+1}}\Big[f(u(x_{j+\frac12},t-t_n;\theta^n))
  -f(u(x_{j-\frac12},t-t_n;\theta^n))\Big]dt,\\ \quad -\infty<j<\infty.\hspace{30pt}
  \endaligned
  \ee
   Since the projection $P^k$ does not change the averages (see ~\eqref{eqinvariantave}) it follows that
   \be\label{eqbalanceexact}\aligned \theta^{n+1}_{j}  -\theta^{n}_{j}= h^{-1}\Big[\int_{I_j}\theta^{n+1}(x)dx-\int_{I_j}\theta^{n}(x)dx\Big]\hspace{50pt} \\
   =-\frac{\lambda}{k}
   \int_{t_n}^{t_{n+1}}\Big[f(u(x_{j+\frac12},t-t_n;\theta^{n}))
   -f(u(x_{j-\frac12},t-t_n;\theta^{n}))\Big]dt,\\ \quad -\infty<j<\infty,\endaligned
   \ee
   where we have used $k=\lambda h.$

   The approximate fluxes $\set{F^{\xi}_{j+\frac12}(t)}$ are consistent of order $q,$ so by ~\eqref{eqdefnconsist} the right-hand side of ~\eqref{eqbalanceexact} satisfies
   \be\label{eqFthetaftheta}\aligned
   -\frac{\lambda}{k}\int_{t_n}^{t_{n+1}}\Big[f(u(x_{j+\frac12},t-t_n;\theta^{n}))
   -f(u(x_{j-\frac12},t-t_n;\theta^{n}))\Big]dt\\
   =-\frac{\lambda}{k}\int_{t_n}^{t_{n+1}}\Big[F^{\theta^{n}}_{j+\frac12}(t-t_n)
   -F^{\theta^{n}}_{j-\frac12}(t-t_n)\Big]dt
   +\mathcal{O}(k^{1+q}),\hspace{25pt}\endaligned
   \ee
   which proves ~\eqref{eqtrunc}.
\end{proof}

   \begin{prop}\label{propconsistaverage} Assume that the approximate fluxes $\set{F^{\xi}_{j+\frac12}(t),\,\,0\leq t <k}_{j=-\infty}^\infty$ are consistent of order $q$ and let $\Phi^k$ be a FVS as in Definition ~\ref{defnapproxsoln}. Let $\theta^n\in V^k$ and \be\label{eqwidetiltheta}\widetilde{\theta^{n+1}}(x)=\Phi^k(\theta^n)\in V^k.\ee
        Let
        $$
        \set{\widetilde{\theta^{n+1}_{j}}=
        h^{-1}\int_{I_j}\widetilde{\theta^{n+1}}(x)dx}_{j=-\infty}^\infty.
        $$
        Then
        \be\label{eqconsistUtildeU}
        |\widetilde{\theta^{n+1}_j}-\theta^{n+1}_j|=\mathcal{O}(k^{1+q}),\quad -\infty<j<\infty.
        \ee
        where the averages $\theta^{n+1}_j$ are as in Proposition ~\ref{propexact}.
   \end{prop}

   \begin{proof}  By ~\eqref{eqapproxsoln} (and the fact that $P^k$ does not change averages) we get
      \be\label{eqapproxsolnunk}\aligned
   \int_{I_j}\Phi^k(\theta^n)dx-\int_{I_j}\theta^ndx=
   -\int_{t_n}^{t_{n+1}}\Big[F^{\theta^n}_{j+\frac12}(t-t_n)-
   F^{\theta^n}_{j-\frac12}(t-t_n)\Big]dt,\\ \quad -\infty<j<\infty.\hspace{20pt} \endaligned
   \ee
   Thus
   \be\label{eqapproxsolnunk1}\widetilde{\theta^{n+1}_j}-\theta^{n}_j=
   -\frac{\lambda}{k}
   \int_{t_n}^{t_{n+1}}\Big[F^{\theta^{n}}_{j+\frac12}(t-t_n)
   -F^{\theta^{n}}_{j-\frac12}(t-t_n)\Big]dt,\quad -\infty<j<\infty.
   \ee
   Comparing this equality with ~\eqref{eqtrunc} we obtain ~\eqref{eqconsistUtildeU}.
   \end{proof}

        We can now compare the consistency result of Proposition ~\ref{propconsistaverage} to the classical consistency definition as recalled in the Introduction (Definition ~\ref{defnconsistfulldisc}).

        Define a discrete time evolution, with $\Delta t=k,$ by

        $$
        \Big[\Phi_{\Delta t,\Delta x}\widetilde{\theta^{n}}\Big]_{j}=\widetilde{\theta^{n+1}_j}-
        \widetilde{\theta^{n}_j},
        \quad -\infty<j<\infty,\quad n=0,1,2,\ldots.
        $$

        \begin{cor}\label{corconsistPhik} Under the assumptions of Proposition ~\ref{propconsistaverage} the discrete operator $\Phi_{\Delta t,\Delta x}$ is of order $q$ in the sense of
        Definition ~\ref{defnconsistfulldisc}. More explicitly, when viewed as acting on the sequence of averages of the \textit{exact} solution, it satisfies Equation ~\eqref{eqdefineconsistdisc}.
        \end{cor}

        \begin{proof}
        It is assumed that the discrete operator acts on the exact solution, namely,  $\widetilde{\theta^{n}_j}=\theta^{n}_j.$ Hence Equation ~\eqref{eqconsistUtildeU} can be rewritten as
         $$
         \Big|\theta^{n+1}_j-[\theta^{n}_j+(\Phi_{\Delta t,\Delta x}\theta^{n})_{j}]\Big|=\mathcal{O}(k^{1+q}),\quad -\infty<j<\infty.
         $$
         This is therefore identical to Equation ~\eqref{eqdefineconsistdisc}.
         \end{proof}

         Thus our definition, while suitable for discontinuous solutions, is in line with the classical definition, when the latter is applicable.









%



%




      \section{\textbf{CONVERGENCE--THE LAX-WENDROFF THEOREM REVISITED}}\label{seclaxwendroff}


          The question of the convergence of  the approximate solutions to a solution of the balance law are discussed in this section.

        Our goal is to impose conditions on the FVS (Definition ~\ref{defnapproxsoln}) that will guarantee the convergence of the approximate solutions to a solution of the balance law (Definition ~\ref{defnbalance}) at a fixed time $t=T$  as $k\to 0.$ Observe that since we are dealing with \textit{systems} and do not assume any entropy condition, we cannot infer that such a solution to the balance law is unique.

        The consistency result of Proposition ~\ref{propconsistaverage} does not imply such convergence. In fact, it deals with the action, over one time step, of the discrete operator on the \textit{exact solution}. In the construction of the approximate solution at time $t_{n+1},$ on the other hand, the operator acts on the \textit{approximate solution} obtained at time $t_n.$   It is given by (see Definition ~\ref{defnapproxsoln})
        \be\label{eqdefapproxsols}
        \widetilde{\theta^{n+1}}(x)=\Phi^k(\widetilde{\theta^n})\in V^k,\quad n=0,1,2,\ldots
        \ee
Thus, the procedure produces errors that accrue at each time step  and do not necessarily vanish at the final time $t=T$ as the time step is refined.

           The above discussion can simply be summarized by saying that both \textit{consistency} and \textit{stability} are needed in order to ensure convergence. In the \textit{linear} case, this is precisely the claim of the celebrated ``Lax equivalence theorem''~\cite{richtmyer}.

           At this stage, it is useful to recall the two main approaches to convergence.
           \begin{itemize}\item ``compactness''--establishing the boundedness of the discrete solutions in a stronger space that is compactly embedded in the expected convergence space. In the case of discontinuous solutions this is universally carried out in total variation spaces.
           \item ``stability''--imposing some boundedness assumptions on the discrete solutions and using consistency in order to control the accumulation of errors.
           \end{itemize}
             The second approach is what can be referred to as the ``Lax-Wendroff methodology''. It necessarily assumes the existence of an exact solution but, on the other hand, the assumptions imposed on the discrete solutions are typically easier to verify in a concrete computation.

              Our study here is in the framework of the second category.

           We remark that in the case of \textit{a linear evolution equation} (even in Banach space) stability (with a suitable assumption on the action of the discrete operator on the residual terms) is sufficient to establish convergence ~\cite{despres}. See also Remark ~\ref{rem:gallouet-consist}.  However there is no similar result that is applicable to the case of interest here, namely, nonlinear hyperbolic balance laws.

           As already mentioned in the Introduction, the special consistency condition ~\eqref{eqconsistlw}, in the context of hyperbolic conservation laws was used in establishing the Lax-Wendroff convergence theorem ~\cite{lax-wendroff}. However, this consistency condition cannot be used in our context of higher order finite volume schemes and was replaced by another notion of consistency (Definition ~\ref{def:consistalpha}).

              Our aim here is to prove that the approximate solutions constructed in ~\eqref{eqdefapproxsols} converge to a solution of the balance law, under certain conditions. The concept of consistency, as developed here, plays a fundamental role in the proof.

    \subsection{\textbf{THE CONVERGENCE THEOREM }}\label{subsecconverge} Fix $T>0.$ Recall  the construction of the discrete (in time) sequence of short time exact solutions
        (see  Claim ~\ref{claimexistentropy} and ~\eqref{eqsuccexact})
                    \be\label{eqdiscbaseexact}\aligned
                    \theta^{n+1}(x)=P^k(u(x,t_{n+1}-t_n;\theta^n))\in V^k, \hspace{80pt}  \\ \quad n=0,1,2,\ldots,N-1,\,\,N=N(k)=k^{-1}T,
                    \endaligned
                    \ee
                    and the sequence of approximate solutions ~\eqref{eqdefapproxsols}
                    \be\label{eqdiscbaseFVS}
                    \widetilde{\theta^{n+1}}(x)=\Phi^k(\widetilde{\theta^n})\in V^k,\quad n=0,1,2,\ldots,N-1.
                    \ee
  For both sequences the initial data is given by taking the projection of the initial function $u_0\in\fU$ on the subspace $V^k$
 \be\label{eqtheta0}
 \theta^0=\widetilde{\theta^0}=P^ku_0\in V^k.
 \ee

                    We assume that the conditions of Definition ~\ref{defnnumericflux} (and in particular the CFL condition) are satisfied, so that approximate fluxes can be constructed. We shall further assume that these fluxes are consistent of order $q>0$ (Definition ~\ref{def:consistalpha}).
         It follows from ~\eqref{eqbalanceexact} that for all grid intervals $I_j,$
         \be\label{eqscheme}\aligned
         \int_{I_j}\theta^{n+1}(x)dx-\int_{I_j}\theta^{n}(x)dx\hspace{125pt}\\
   =
   -\int_{t_n}^{t_{n+1}}\Big[f(u(x_{j+\frac12},t-t_n;\theta^{n}))
   -f(u(x_{j-\frac12},t-t_n;\theta^{n}))\Big]dt,\\ \quad -\infty<j<\infty,\endaligned
   \ee
and from ~\eqref{eqapproxsoln} that for all grid intervals $I_j,$










\be\label{eqscheme1}\aligned
       \int_{I_j}[\widetilde{\theta^{n+1}}(x)-\widetilde{\theta^{n}}(x)]dx\hspace{150pt}
       \\=-\int_{t_n}^{t_{n+1}}[F^{\widetilde{\theta^{n}}}_{j+\frac12}(t-t_n)-
       F^{\widetilde{\theta^{n}}}_{j-\frac12}(t-t_n)]dt,
       \,\,-\infty<j<\infty.
       \endaligned
       \ee


%


%







       We construct a function $\widetilde{\Upsilon^k}(x,t)$ as follows.
      \be\label{eqdefnvkxt}
      \aligned
      \widetilde{\Upsilon^k}(x,t)=\frac{1}{k}[(t_{n+1}-t)\widetilde{\theta^n}(x)+(t-t_{n})\widetilde{\theta^{n+1}}(x)],\quad t\in [t_{n},t_{n+1}],\,\,\,\\ n=0,1,\ldots,N-1.
      \endaligned
      \ee
         Observe that $t_n=nk$ depends on $k.$

          Instead of the classical Lax-Wendroff theorem we get here the following theorem.



%








\begin{thm}\label{thmconverge} Assume that the FVS ~\eqref{eqdiscbaseFVS} is consistent of order $q>0.$ Let $\set{k_m\downarrow 0}$ be a decreasing sequence of time steps. Let $u_0\in \fU$ (see ~\eqref{eqdeffU})  and let $\set{\widetilde{\Upsilon^{k_m}}(x,t)}_{m=1}^\infty$ be the corresponding functions defined in
          ~\eqref{eqdefnvkxt}.

          Suppose that

          \begin{enumerate}
          \item[(i)] The sequence $\set{\widetilde{\Upsilon^{k_m}}(x,t)}_{m=1}^\infty$ is uniformly bounded in $  L^\infty([0,T],L^\infty(\RR)). $
              \item[(ii)] The sequence $\set{\widetilde{\Upsilon^{k_m}}(x,t)}_{m=1}^\infty$ converges in $C([0,T],L^1_{loc}(\RR))$ to a function $v(x,t)$ (in particular it is uniformly bounded in this space).
          \end{enumerate}
            Then $v(x,t)$ is a  solution of the balance law ~\eqref{eqbalancecons} in  $\RR\times [0,T].$
          \end{thm}

          \begin{rem}
             The boundedness and convergence hypotheses in the theorem can be formulated in terms of the discrete solutions $\widetilde{\theta^n}(x)$ as follows, where  $N_m=k_m^{-1}T.$
             \begin{itemize}\item The set $\set{\set{\widetilde{\theta^n}(x)}_{n=1}^{N_m}\subseteq V^{k_m}}_{m=1}^\infty$ is uniformly bounded in $\fU $\,\,(in the topology ~\eqref{eqnormU}).
            \item There exists a function $v(\cdot,t)\in C([0,T],L^1_{loc}(\RR))$ so that
             \be\label{eqthetanvxt}
              \lim\limits_{m\to\infty}\sup\limits_{1\leq n\leq N_m}d(\widetilde{\theta^n}(x),v(x,t_n))=0,
             \ee
             where the metric $d(y,z)$ is given in ~\eqref{eqL1locmetric}.
             \end{itemize}
          \end{rem}

          \begin{proof}[Proof of Theorem ~\ref{thmconverge}]  We have, in view of ~\eqref{eqscheme1}, for $n=0,1,2,\ldots,N_m-1,$
          \be\label{eqscheme2}\aligned
       \int_{I_j}[\widetilde{\Upsilon^{k_m}}(x,t_{n+1})-\widetilde{\Upsilon^{k_m}}(x,t_{n})]dx\hspace{80pt}
       \\=-\int_{t_n}^{t_{n+1}}[F^{\widetilde{\Upsilon^{k_m}}(x,t_{n})}_{j+\frac12}(t-t_n)-
       F^{\widetilde{\Upsilon^{k_m}}(x,t_{n})}_{j-\frac12}(t-t_n)]dt,\\
       \,\,-\infty<j<\infty.
       \endaligned\ee

           By the convergence assumption the set
           $$
           H=\set{\widetilde{\Upsilon^{k_m}}(x,t),\,\,t\in[0,T],\,\,m=1,2,\ldots}
          $$
           is admissible in the sense of Definition ~\ref{def:initialcompact}, so  the consistency condition
                ~\eqref{eqdefnconsist} entails
               \be\label{eqreplaceFf}\aligned
               \int_{t_n}^{t_{n+1}}[F^{\widetilde{\Upsilon^{k_m}}(x,t_{n})}_{j+\frac12}(t-t_n)-
       F^{\widetilde{\Upsilon^{k_m}}(x,t_{n})}_{j-\frac12}(t-t_n)]dt\hspace{100pt}\\=
       \int_{t_n}^{t_{n+1}}[f(u(x_{j+\frac12},t-t_n;\widetilde{\Upsilon^{k_m}}(x,t_{n})))-
       f(u(x_{j-\frac12},t-t_n;\widetilde{\Upsilon^{k_m}}(x,t_{n})))]dt\\ +\mathcal{O}(k_m^{2+q}),\hspace{270pt}
       \endaligned\ee
        where the notation ~\eqref{eqdefnthetan}
       \be
       u(\cdot,t-t_n;\widetilde{\Upsilon^{k_m}}(\cdot,t_{n}))=[S(t-t_n)\widetilde{\Upsilon^{k_m}}(\cdot,t_{n})]
       \ee
       has been used.

       Here and below we use $\mathcal{O}(k_m^{2+q})$ to designate a remainder that satisfies
           $$ |\mathcal{O}(k_m^{2+q})|\leq Ck_m^{2+q},$$
           where $C>0$ depends on $H$ but is independent of $j,\,n,\,m.$

       Inserting ~\eqref{eqreplaceFf} in ~\eqref{eqscheme2} and summing over $n$ time steps yields

        \be\label{eqscheme3}\aligned
       \int_{I_j}[\widetilde{\Upsilon^{k_m}}(x,t_{n})-\widetilde{\Upsilon^{k_m}}(x,0)]dx\hspace{200pt}
      \\= -\suml_{l=0}^{n-1}\int_{t_l}^{t_{l+1}}\Big[f([S(t-t_l)\widetilde{\Upsilon^{k_m}}(\cdot,t_{l})](x_{j+\frac12}))-
       f([S(t-t_l)\widetilde{\Upsilon^{k_m}}(\cdot,t_{l})](x_{j-\frac12}))\Big]dt\\ +\mathcal{O}(k_m^{1+q}).\hspace{280pt}
       \endaligned\ee
           Take an interval $[a,b]\subseteq\RR$ and let $J^m_1<J^m_2$ be indices such that, with $h_m=\lambda^{-1}k_m,$
         $$
           J^m_1h_m=a,\quad J^m_2h_m=b.
           $$
           Summing in ~\eqref{eqscheme3} over $J^m_1\leq j\leq J^m_2$ we get
           \be\label{eqscheme4}\aligned
       \int\limits_a^b[\widetilde{\Upsilon^{k_m}}(x,t_{n})-\widetilde{\Upsilon^{k_m}}(x,0)]dx\hspace{170pt}
      \\= -\suml_{l=0}^{n-1}\int_{t_l}^{t_{l+1}}\Big[f([S(t-t_l)\widetilde{\Upsilon^{k_m}}(\cdot,t_{l})](b))-
       f([S(t-t_l)\widetilde{\Upsilon^{k_m}}(\cdot,t_{l})](a))\Big]dt\\+\mathcal{O}(k_m^{q}).\hspace{260pt}
       \endaligned
       \ee
             Take $n=N_m$ so that $t_n=T.$

             We consider the limits, as $m\to\infty,$ of the two sides in Equation ~\eqref{eqscheme4}.
   Using the convergence assumptions we obtain readily

             \be\label{eqvbalancesoln}
             \lim\limits_{m\to\infty}\int\limits_a^b[\widetilde{\Upsilon^{k_m}}(x,t_{n})-\widetilde{\Upsilon^{k_m}}(x,0)]dx
              =\int\limits_a^b[v(x,T)-v(x,0)]dx.
              \ee
              The integral in the right-hand side of ~\eqref{eqscheme4} is more complicated. Observe that by the boundedness and convergence hypotheses of the theorem, for every interval $[c,d]\subseteq\RR,$
              \be\label{eqlimrhside}
              \lim\limits_{m\to\infty}\int_c^d\suml_{l=0}^{n-1}\int_{t_l}^{t_{l+1}}f(S(t-t_l)\widetilde{\Upsilon^{k_m}}(\cdot,t_{l})(x))
       dt\,dx=\int_c^d\int_0^T f(v(x,t))dt\,dx.
       \ee
         In fact, denoting $g_m(x)=\int_0^T\suml_{l=0}^{n-1}f(S(t-t_l)\widetilde{\Upsilon^{k_m}}(\cdot,t_{l})(x))dt,\,\,g(x)=\int_0^Tf(v(x,t))dt,$ the hypotheses yield
      \be\label{eqgmtov}\lim\limits_{m\to\infty} g_m(x)=g(x)\quad \mbox{in}\,\,\,L^1([c,d]).
         \ee
         Let $c,\,d\in [a,b]$ be two Lebesgue points of $g(x).$ Inserting the limits ~\eqref{eqvbalancesoln} (with $[a,b]$ replaced by $[c,d]$) and
         ~\eqref{eqgmtov} in ~\eqref{eqscheme4} we get
         \be \int\limits_c^d[v(x,T)-v(x,0)]dx=-(g(d)-g(c))=-\int_0^T [f(v(d,t))-f(v(c,t))]dt.
         \ee
            In particular, it follows that
            \be |g(d)-g(c)|\leq C |d-c|,\quad C=2\sup\limits_{(x,t)\in[a,b]\times[0,T]}|v(x,t)|.
            \ee
            Thus $g(x)=\int_0^Tf(v(x,t))dt$ is Lipschitz continuous in $\RR$ and for any $[a,b]\subseteq \RR$
            \be
            \int\limits_a^b[v(x,T)-v(x,0)]dx=-\int_0^T [f(v(b,t))-f(v(a,t))]dt.
         \ee
            We conclude that $v(x,t)$ satisfies the requirements of Definition ~\ref{defnbalance} and is a solution to the balance law, as asserted.

             %







          \end{proof}

          It was shown (Example ~\ref{examplegodunov1}) that the Godunov approximate fluxes on piecewise-constant functions are consistent of any order while the GRP upgrading (Definition  ~\ref{defGRPflux}) is consistent of  (at least) first order, hence the following corollary holds.

          \begin{cor}[\textbf{Godunov, GRP and MUSCL convergence}]\label{corGRPconverge}\begin{enumerate}
           \item[(i)] Let the FVS ~\eqref{eqdiscbaseFVS} be given by the approximate Godunov fluxes (Example ~\ref{examplegodunov1}), confined to piecewise-constant functions. Then the limit of any convergent sequence, subject to the hypotheses  (i)-(ii) of Theorem ~\ref{thmconverge}, is a solution to the balance law.
           \item[(ii)] Let the FVS ~\eqref{eqdiscbaseFVS} be given by the approximate GRP fluxes ~\eqref{equpgradeGodflux}, where the space $V^k$ of approximating functions can be of any finite order. Then the limit of any convergent sequence, subject to the hypotheses  (i)-(ii) of Theorem ~\ref{thmconverge}, is a solution to the balance law.
               \item[(iii)] Let the FVS ~\eqref{eqdiscbaseFVS} be given by the approximate MUSCL fluxes  (Example ~\ref{exampleMUSCL}), where the space $V^k$ of approximating functions can be of any finite order. Then the limit of any convergent sequence, subject to the hypotheses  (i)-(ii) of Theorem ~\ref{thmconverge}, is a solution to the balance law.
  \end{enumerate}
   \end{cor}

          \begin{rem}\label{remconvergeq=0} A fundamental assumption in Theorem ~\ref{thmconverge} is that the approximate fluxes are consistent of order $q>0.$ Recalling the discussion in Subsection ~\ref{subsubGRP} it follows that when the Godunov approximate fluxes are implemented  for piecewise-constant functions,any limit function is a solution to the balance law, as stipulated by Corollary ~\ref{corGRPconverge}. On the other hand, taking a (spatially) second-order approximation (namely, piecewise-linear functions), and still using the Godunov approximate fluxes ~\eqref{eqGodunovpiecelinear}, the order of consistency is $q=0.$ The convergence theorem is not applicable and convergence may fail. On the other hand, as is stated in Corollary ~\ref{corGRPconverge},  implementing the GRP or MUSCL fluxes raises the order to $q=1$ and ensures that any limit function is a solution of the balance law.

          These considerations are convincingly demonstrated in the numerical examples worked out in ~\cite{Jiequan-Yue}, where the aforementioned two possibilities for approximate fluxes (with piecewise-linear data) were tested (see Figs. 6.2 and 6.5 there).
          \end{rem}

          \subsection{\textbf{A MEASURE THEORY LEMMA }}\label{subsecmeasure}

           Throughout the rest of this section we fix a $T>0.$ The time steps to be considered will be of size $k=\frac{1}{N}T$ for an integer $N>1.$

           Our final goal (Corollary ~\ref{coraverageconv}) is to prove that the \textit{grid averages} of the approximate solutions converge to the solution of the balance law obtained in Theorem ~\ref{thmconverge}. In proving this, some basic measure-theoretic facts are established.















\begin{defn}\label{defnYkav}
    Given the spacetime grid $\GRST$ ~\eqref{eqkgrid} and a function $Y(x,t)\in L^1_{loc}(\RR\times\RplusT),$ we denote by $Y^{k,av}(x,t)$ the space-averaged  function that consists of the averages in grid intervals,
      \be\label{eqYkav}Y^{k,av}(x,t)=\frac{1}{h}\int\limits_{I_j}Y(z,t)dz,\quad (x,t)\in I_j\times[0,T],
      \ee
         $$
         I_j=(x_{j-\frac12},x_{j+\frac12}),\quad -\infty<j<\infty.
         $$
         \end{defn}

         We have
         $$
         \int\limits_{I_j}Y^{k,av}(x,t)dx=\int\limits_{I_j}Y(x,t)dx
         $$
         and
         $$
         \Big|\int\limits_{I_j}Y^{k,av}(x,t)dx\Big|=
         \int\limits_{I_j}|Y^{k,av}(x,t)|dx.$$
         So integrating  over any $[a,b]\subseteq [0,T]$ yields
         \be\label{eqintYrectangle}
         \int\limits_{I_j\times[a,b]}|Y^{k,av}(x,t)|dxdt\leq
         \int\limits_{I_j\times[a,b]}|Y(x,t)|dxdt.
         \ee
          It follows that if the bounded set $K\subseteq\RR\times\RplusT$ is a union of such rectangles then
          \be\label{eqintYavY}
          \int\limits_K|Y^{k,av}(x,t)|dxdt\leq \int\limits_K|Y(x,t)|dxdt.
          \ee
          Using a density argument we now obtain:
          \begin{claim}\label{claimYavconvY}
          Let $Y(x,t)\in L^1_{loc}(\RR\times\RplusT),$ then for every bounded $K\subseteq\RR\times\RplusT,$
          \be\label{eqYavconvYloc}
         \lim_{k\to 0} \int\limits_K|Y^{k,av}(x,t)-Y(x,t)|dxdt=0.
          \ee
          \end{claim}

       Claim ~\ref{claimYavconvY} entails the following lemma.

      \begin{lem}\label{lemstrongconvavg}
         Let $\set{w_{m}(x,t)}_{m=1}^\infty\subseteq C(\RplusT,L^1_{loc}(\RR))$ be a sequence of functions that converges to a function $w(x,t)$ in the sense that
        $$
        \lim\limits_{m\to\infty}w_{m}(\cdot,t)=w(\cdot,t),\quad\mbox{in}\,\,C(\RplusT,L^1_{loc}(\RR)).
        $$
Let $\set{k_m\downarrow 0}$ be a decreasing sequence.
       Then the sequence of the corresponding average functions $\set{w^{k_m,av}_{m}(x,t)}_{m=1}^\infty$ converges  to $w(x,t)$ in $L^1_{loc}(\RR\times\RplusT).$

      \end{lem}

        \begin{proof}
        Let $K\subseteq \RR\times\RplusT$ be bounded. Then in view of ~\eqref{eqintYavY}
        $$
        \lim\limits_{m\to\infty}\int\limits_K |w^{k_m,av}_{m}(x,t)-w^{k_m,av}(x,t)|dxdt=0,
        $$
        and in view of ~\eqref{eqYavconvYloc}
        $$\lim\limits_{m\to\infty}\int\limits_K |w^{k_m,av}(x,t)-w(x,t)|dxdt=0.$$
        \end{proof}

        \begin{cor}\label{coraverageconv} Assume the conditions of Theorem ~\ref{thmconverge} and define the sequence of piecewise-constant functions
        \be\label{eqUpsilonkav}
        \widetilde{\Upsilon^{k_m,av}}(x,t)=\frac{1}{h}\int\limits_{I_j}\widetilde{\Upsilon^{k_m}}(z,t)dz,\quad (x,t)\in I_j\times[0,T],
        \ee
         $$I_j=(x_{j-\frac12},x_{j+\frac12}),\quad -\infty<j<\infty.
         $$
         Then the sequence  $\set{\Upsilon^{k_m,av}_{m}(x,t)}_{m=1}^\infty$ converges  to $v(x,t)$ in $L^1_{loc}(\RR\times\RplusT).$
        \end{cor}

           This corollary is of great practical significance, as it states that the solution to the balance law can be recovered from the cell averages of the approximate solutions. Obviously, these averages are easier to obtain, in the computational procedure, than the full (piecewise-polynomial) approximate solutions.




%


%












          \section{\textbf{GODUNOV COMPATIBILITY AND ENTROPY  }}\label{subsecentropy} A major difficulty in the theory of nonlinear balance laws is that the limiting solutions obtained in Theorem ~\ref{thmconverge} need not be unique.   Recall that a solution to the balance law (Definition ~\ref{defnbalance}) is necessarily a weak solution to the corresponding conservation law  ~\eqref{eqconslaw}. The entropy condition (see Subsection ~\ref{subsecscalarlaw} above), essentially the only tool available for establishing uniqueness, has been applicable only in the scalar case ~\cite{dafermos,GodlewskiRaviart} (and some $2\times 2$ systems ~\cite{LiuTP}).

          It is well-known that the Godunov FVS $\Phi^{k,G}$ provides a ``reference frame'' to full classes of (first order, scalar) approximate solutions, for example to all ``$E-$schemes''~\cite[Chapter 3, Lemma 4.1]{GodlewskiRaviart}. Also, for a class of $2\times 2$ systems it was shown in ~\cite{diperna} that all limits of approximate solutions obtained by the Godunov scheme are entropy solutions.   These observations seem to justify the introduction of  the concept of ``Godunov-compatible'' schemes (Definition ~\ref{defineGodtype}). It is based on the Assumption ~\ref{assumeGodunov}  that the Godunov FVS converges to a unique entropy solution. it is then shown  that the approximate solutions produced by Godunov-compatible schemes converge to the same solution.

          The treatment here may be compared to that of the Glimm scheme: Under suitable conditions all weak solutions obtained as limits satisfy the entropy condition ~\cite[Theorem 2.2]{lax1971} and the solution is unique in the class of approximate solutions obtained by the front tracking method ~\cite{bressan} (see also ~\cite[Chapter XIV]{dafermos}).

          Let us first consider the Godunov FVS as introduced  in Example ~\ref{examplegodunov1} (for general systems, not only scalar).  The scheme is used  for first order (namely, piecewise constant) spaces $V^k.$  The notation $\Phi^{k,G}$ is used for the Godunov FVS.

           In this case the  FVS  yields a discrete sequence $\set{\widetilde{\theta^{n,G}}(x)}_{n=0}^N$ of piecewise constant functions as in ~\eqref{eqdefapproxsols}. Thus
 \be\label{eqdefthetanG}
 \widetilde{\theta^{n,G}}(x)=\widetilde{\theta^{n,G}_j},\,\,x\in I_j,\,\,-\infty<j<\infty.
 \ee


          The values $\set{\widetilde{\theta^{n,G}_j}}_{j=-\infty}^\infty $ satisfy (see ~\eqref{eqbalanceexact})
          \be\label{eqbalanceexactG}
          \aligned
           \widetilde{\theta^{n+1,G}_{j}}  -\widetilde{\theta^{n,G}_{j}}\hspace{240pt}
  \\  =-\lambda
   \Big[f(u^G(x_{j+\frac12},t-t_n;\widetilde{\theta^{n,G}}))
   -f(u^G(x_{j-\frac12},t-t_n;\widetilde{\theta^{n,G}}))\Big], -\infty<j<\infty,
   \endaligned
   \ee
   where $u^G(x_{j+\frac12},t-t_n;\widetilde{\theta^{n}})\equiv const$ is the solution to the Riemann problem at $x=x_{j+\frac12}.$

    As in Equation ~\eqref{eqdefnvkxt} we can now use the set $\set{\widetilde{\theta^{n,G}}(x)}_{n=0}^N$ in order to define the function $ \widetilde{\Upsilon^{k,G}}(x,t)$ for the Godunov scheme.

\subsection {\textbf{THE SCALAR GODUNOV SCHEME}}
   It is well-known that \textbf{ in the scalar case} the  Godunov scheme possesses the same boundedness and contraction properties as the exact solution to the balance law (Claim ~\ref{claimStproperty}):
   \begin{claim}[\textbf{\cite[Section 3.3]{GodlewskiRaviart}}] The sequence of solutions to the Godunov scheme satisfies
   \begin{itemize}
   \item
    $$
    \|\widetilde{\theta^{n+1,G}}\|_\infty\leq \|\widetilde{\theta^{n,G}}\|_\infty,\quad n=0,1,2,\ldots, N-1.
    $$
   \item If $\widetilde{\theta^{0,G}}$ and $\widetilde{\chi^{0,G}}$ are two piecewise constant functions and
   $\set{\widetilde{\theta^{n,G}}(x)}_{n=0}^N,$  $\set{\widetilde{\chi^{n,G}}(x)}_{n=0}^N$ are the corresponding solutions by the Godunov scheme, then
       $$
       \|\widetilde{\theta^{n+1,G}}-\widetilde{\chi^{n+1,G}}\|_1\leq \|\widetilde{\theta^{n,G}}-\widetilde{\chi^{n,G}}\|_1,\quad n=0,1,2,\ldots, N-1.
       $$
   \end{itemize}
   \end{claim}

      The fact that the approximate solutions derived by the Godunov scheme converge to the unique entropy solution is a fundamental fact of the theory of discretization of (scalar) conservation laws:

      \begin{claim}[\textbf{\cite[Chapter 3, Theorem 4.1]{GodlewskiRaviart}}] Assume that the CFL condition is satisfied and also that the initial function $u_0$ has finite total variation. Then the limit
      $$
      v(x,t)=\lim\limits_{k\to 0}\widetilde{\Upsilon^{k,G}}(x,t)
      $$
      exists in $C([0,T],L^1_{loc}(\RR))$ and is the unique entropy solution of the balance law.

      \end{claim}
      \begin{rem}
      Note that the condition that $u_0$ has finite total variation is not really needed ~\cite[Theorem 29.2]{eymard-gallouet}.
      \end{rem}

      \subsection {\textbf{THE CASE OF SYSTEMS}}
    Recall ~\eqref{eqYkav}) that the projection of $\xi\in V^k$ on the space of piecewise constant functions, namely, the set of averages in the grid intervals, is designated as
   \be\label{eqxikav}
   \xi^{k,av}(x)=h^{-1}\int_{I_j} \xi(z) dz,\quad x\in I_j=(x_{j-\frac12},x_{j+\frac12}),\,\,\,\,-\infty<j<\infty,
   \ee
    and satisfies (compare ~\eqref{eqintYrectangle})
    \be\label{eqxikvlessxi}
    \int\limits_{\RR}|\xi^{k,av}(x)|dx\leq \int\limits_{\RR}|\xi(x)|dx.
    \ee
   For general systems  consider the Godunov scheme and recall (Corollary ~\ref{corGRPconverge}) that under suitable hypotheses the approximate solutions converge to a solution of the balance law. We now impose the following \textbf{fundamental hypothesis on the Godunov scheme} regarding the uniqueness of these solutions.

    \begin{assume}\label{assumeGodunov}[\textbf{Godunov Scheme}] Let $\mathfrak{B}_K$ be the ball of radius $K>0$  in  $ \fU$ (see ~\eqref{eqdeffU}) and let $u_0\in \mathfrak{B}_K.$ Let
     $\theta^{0,G}=u_0^{k,av}.$ The Godunov scheme $\Phi^{k,G},$ applied to $\theta^{0,G}=u_0^{k,av}$ converges to a unique solution of the balance law. More precisely, if $\Phi^{k,G}$ is the FVS in Theorem ~\ref{thmconverge} then, under the hypotheses of the theorem, all limits of  subsequences $ \widetilde{\Upsilon^{k_m,G}}(x,t)$ obtained in the theorem are identical.

     Furthermore, if $v_0\in\mathfrak{B}_K$ is another initial function and $\psi^{0,G}=v_0^{k,av}$ ,  then
      \be\label{eqcontGodscheme}
         \|\Phi^{k,G}\theta^{0,G}-\Phi^{k,G}\psi^{0,G}\|_1\leq(1+Ck)\|\theta^{0,G}-\psi^{0,G}\|_1,
      \ee
      where $C>0$ depends only on $K.$
    \end{assume}

     Consider a space $V^k$ (of any order) and an FVS as in ~\eqref{eqdiscbaseFVS}.  In view of ~\eqref{eqinvariantave} and ~\eqref{eqapproxsoln} the map $\Phi^k$ is conservative:
   \be\label{eqPhikconserv}
   \int_\RR \Phi^k\xi(x) dx=\int_\RR\xi(x) dx,\quad \xi\in V^k.
   \ee

   Let $\set{\widetilde{\theta^n}\in V^k}_{n=0}^N$ be the discrete set of approximate solutions, as constructed in ~\eqref{eqdiscbaseFVS}. From them
        we obtain the function $\widetilde{\Upsilon^{k}}(x,t)$
          defined in ~\eqref{eqdefnvkxt} and the ``average function''
            $\widetilde{\Upsilon^{k,av}}(x,t)$ as in ~\eqref{eqYkav}.
            The set of cell averages of $\widetilde{\theta^n}$ is denoted by $\set{\widetilde{\theta^{n}_{j}}=
        h^{-1}\int_{I_j}\widetilde{\theta^{n}}(x)dx}_{j=-\infty}^\infty.$

    The following definition encapsulates the meaning of the FVS $\Phi^k$ as being compatible with the Godunov scheme.
   \begin{defn}\label{defineGodtype}[\textbf{Godunov Compatibility}] The FVS $\Phi^k$ (consistent of order $q>0$) is \textit{compatible with the Godunov scheme} if the following conditions hold.
    \begin{enumerate}
    \item[(i)]  The FVS $\Phi^k$ coincides with the Godunov scheme on piecewise constant functions; if $\xi\in V^k$ is piecewise constant then
         \be \Phi^k\xi=\Phi^{k,G}\xi.
         \ee
         \item[(ii)] Let $H$ be an admissible set (Definition ~\ref{def:initialcompact}). Then
          \be\label{eqPhiPhikG}
           \int\limits_{\RR}|\Phi^k\xi(x)-\Phi^{k,G}\xi^{k,av}(x)|dx=o(k),\quad \xi\in H,
          \ee
          where $o(k)$ is uniform for all $\xi\in H.$
    \end{enumerate}
    \end{defn}

    \begin{rem}\label{remGodtypeii}[\textbf{Explaining ~\eqref{eqPhiPhikG}}] This remark is intended as a {\upshape{motivation}}, certainly not a proof, for ~\eqref{eqPhiPhikG}.

    Let $\eta(x)=\Phi^k\xi(x)$ and $\chi(x)=\Phi^{k,G}\xi^{k,av}(x).$
    Observe that in light of ~\eqref{eqdefnconsist} and ~\eqref{eqtrunc}, for a fixed index $j\in \mathbb{Z},$
    \be\label{eqthetaapp}\aligned
    \int_{I_j}[\eta(x)-\xi(x)]dx=
   -\int\limits_{0}^{k}[F^{\xi}_{j+\frac12}-F^{\xi}_{j-\frac12}]dt\hspace{40pt}\\
    =-\int_0^k\Big[f(u(x_{j+\frac12},t;\xi))
   -f(u(x_{j-\frac12},t;\xi))\Big]dt+\mathcal{O}(k^{2+q})\\
   =\int_{I_j}[u(x,k;\xi)(x)-\xi(x)]dx+\mathcal{O}(k^{2+q}).\hspace{73pt}
  \endaligned\ee
  (See Assumption ~\ref{assumesoln} for the definition of $u(x,t;\xi)$).

  Since the Godunov scheme yields the exact mean value,
  \be\label{eqchiapp}
  \int_{I_j}[\chi(x)-\xi^{k,av}(x)]dx=\int_{I_j}[u(x,k;\xi^{k,av})(x)-\xi(x)]dx.
  \ee
       Subtracting ~\eqref{eqchiapp} from ~\eqref{eqthetaapp} yields
       \be\label{eqeta-chi}
       \int_{I_j}(\eta(x)-\chi(x))dx=\int_{I_j}[u(x,k;\xi)(x)-u(x,k;\xi^{k,av})(x)]dx+\mathcal{O}(k^{2+q}).
       \ee
      Assuming that the exact solution $u$ is \textit{a contraction in $L^1$,} in conjunction with the \textit{finite propagation speed property} leads to
       \be\label{eqcontractetachi}
       \int_{I_j}|\eta(x)-\chi(x)|dx\leq C_1\suml_{m=j-l}^{j+l}\int_{I_m}|\xi(x)-\xi^{k,av}(x)|dx+\mathcal{O}(k^{2+q}),
       \ee
       where $C_1>0$ and the integer $l\geq 1$ are independent of $j.$

       If the projection $P^k$ (see ~\eqref{eqinvariantave}) involves a \textit{slope limiter}, then the right-hand side of ~\eqref{eqcontractetachi} is estimated by $C_2k^2(1+\mathcal{O}(k^q)),$ and summation over $j$ leads to
       \be
       \int_{\RR}|\eta(x)-\xi(x)|dx\leq C_2k(1+\mathcal{O}(k^q)).
       \ee
       Thus, ~\eqref{eqPhiPhikG} can be understood as assuming that in at most $o(k^{-1})$ grid cells there is a significant discrepancy between $\Phi^k\xi$ and the piecewise-constant function $\Phi^{k,G}\xi^{k,av}$ obtained by application of the Godunov scheme.
    \end{rem}

    \begin{thm}\label{thmunique} Assume the validity of Assumption ~\ref{assumeGodunov} and that the FVS $\Phi^k$  is consistent of order $q>0$ and  compatible with the Godunov scheme.  Let $\set{\widetilde{\theta^n}\in V^k}_{n=0}^N$ ($N=k^{-1}T$)  be the discrete set of approximate solutions, as constructed in ~\eqref{eqdiscbaseFVS}. Then the limit function obtained in Theorem ~\ref{thmconverge} is unique, namely, under the hypotheses of the theorem there is a unique limit function for all converging subsequences.

    \end{thm}

    \begin{proof} Let $H=\set{\set{\widetilde{\theta^n}(x)}_{n=1}^{N_m}\subseteq V^{k_m}}_{m=1}^\infty$ be the convergent sequence, as in  Theorem ~\ref{thmconverge}.      The initial function $\widetilde{\theta^0}$ is given in ~\eqref{eqtheta0}.
        The idea of the proof is to compare the evolving  sequence $\set{\widetilde{\theta^n}(x)}$ with the  evolution of the Godunov scheme $\set{\widetilde{\theta^{n,G}}(x)}.$








    Let $(\widetilde{\theta^{n}})^{k_m,av}$ be the piecewise-constant  function consisting of the cell averages values of $\widetilde{\theta^{n}}$ (see ~\eqref{eqxikav}), and let
    $\widetilde{\psi^{n+1}}=\Phi^{k_m}(\widetilde{\theta^{n}})^{k_m,av}=\Phi^{k_m,G}(\widetilde{\theta^{n}})^{k_m,av}$ (see Definition ~\ref{defineGodtype}(i)).
    We have also the sequence ~\eqref{eqdefthetanG} $\set{\widetilde{\theta^{n,G}}(x)}$  obtained by the Godunov scheme, with $\widetilde{\theta^{0,G}}=(\widetilde{\theta^0})^{k_m,av}.$

     In view of ~\eqref{eqcontGodscheme} we have
    \be\label{eqdiffpsithetaav}  \|\widetilde{\theta^{n+1,G}}-\widetilde{\psi^{n+1}}\|_1 =\|\Phi^{k_m,G}\theta^{n,G}-\Phi^{k_m,G}(\widetilde{\theta^{n}})^{k_m,av}\|_1\leq(1+Ck_m)\|\widetilde{\theta^{n,G}}-(\widetilde{\theta^{n}})^{k_m,av}\|_1,
      \ee
    and invoking ~\eqref{eqxikvlessxi} leads to
    \be\label{eqdiffpsithetaav1}  \|\widetilde{\theta^{n+1,G}}-\widetilde{\psi^{n+1}}\|_1 \leq(1+Ck_m)\|\widetilde{\theta^{n,G}}-\widetilde{\theta^{n}}\|_1,
      \ee

    Since the set $H$ is admissible, the Godunov compatibility (Definition ~\ref{defineGodtype}(ii)) implies that, for any given $\eps>0,$ for $m>1$ sufficiently large,
        \be\label{eqdiffPhiGod}
    \| \widetilde{\theta^{n+1}}-\widetilde{\psi^{n+1}}\|_1\leq \eps k_m.
    \ee
    Combining ~\eqref{eqdiffpsithetaav1} and ~\eqref{eqdiffPhiGod} we get
    \be\label{eqestthetan+1} \|\widetilde{\theta^{n+1,G}}-\widetilde{\theta^{n+1}}\|_1 \leq \eps k_m+(1+Ck_m)\|\widetilde{\theta^{n,G}}-(\widetilde{\theta^{n}})^{k_m,av}\|_1.
    \ee
      It follows that
      $$\|\widetilde{\theta^{n+1,G}}-\widetilde{\theta^{n+1}}\|_1\leq \eps k_m\suml_{r=0}^n (1+Ck_m)^r+(1+Ck_m)^n \|\widetilde{\theta^{0,G}}-\widetilde{\theta^0}\|_1,$$
      and since $n\leq N_m=T (k_m)^{-1},$
      \be\label{eqfinalest}
      \|\widetilde{\theta^{n+1,G}}-\widetilde{\theta^{n+1}}\|_1\leq\eps k_m\frac{(1+Ck_m)^{N_m}-1}{Ck_m}+e^{CT}\|\widetilde{\theta^{0,G}}-\widetilde{\theta^0}\|_1.
      \ee
      By Assumption ~\ref{assumeGodunov} the sequence produced by the Godunov scheme $\widetilde{\theta^{n,G}}$ (or, rather the sequence of functions  $\widetilde{\Upsilon^{k_m,G}}(x,t)$) converge to a unique limit. Since $\eps>0$ is arbitrary this limit must be $v(x,t),$ the limit of $ \widetilde{\Upsilon^{k_m}}(x,t),$
       which is therefore unique.
\end{proof}

\begin{rem}\label{remgodunoventropy}[\textbf{Godunov Scheme and Entropy}]
  Systems that allow for entropy/ entropy-flux formulations play a special role in the study of balance laws. This is true in particular in various (hyperbolic) models of fluid dynamics. In such cases, Assumption ~\ref{assumeGodunov} can be relaxed, requiring only that all possible limit functions are entropy solutions. As is well-known, this requirement is not sufficient to ensure uniqueness. However, in this case  Theorem ~\ref{thmunique}
 can be modified (under the same hypotheses) to state that all possible limits obtained by the FVS $\Phi^k$ are entropy solutions.
\end{rem}

\begin{example}\label{examisentropy}[\textbf{Isentropic Gas Dynamics}]
Consider the Euler system of compressible, isentropic flow in one space dimension:
  \be\label{eqgasisent}\aligned
      \rho_t+(\rho u)_x=0,\\
      (\rho u)_t+(\rho u^2+p(\rho))_x=0,
      \endaligned\ee
      subject to initial conditions
      $$\rho(x,0)=\rho_0(x)\geq 0,\quad u(x,0)=u_0(x),\quad x\in\RR.
      $$
      Here $\rho$ is the density, $u$ is the velocity and the gas is polytropic: $p=k\rho^\gamma$ with $1<\gamma\leq \frac53.$
      Then we have the following corollary to Theorem ~\ref{thmunique}.
\end{example}

\begin{cor}\label{corentropysoln}
 Suppose that the FVS $\Phi^k$  is consistent of order $q>0$ and  compatible with the Godunov scheme.  Let $\set{\widetilde{\theta^n}\in V^k}_{n=0}^N$ ($N=k^{-1}T$)  be the discrete set of approximate solutions, obtained by applying $\Phi^k$ to the system ~\eqref{eqgasisent}. Then all limit functions obtained in Theorem ~\ref{thmconverge} are entropy solutions of the system.
\end{cor}

\begin{proof} It is shown in ~\cite{ding-chen} that, under some additional conditions on the initial data, the approximate solutions obtained by the Godunov scheme converge to entropy solutions of the system.
\end{proof}


\begin{thebibliography}{100}



\bibitem{BenArtzi-Falcovitz-84} M. Ben-Artzi and J. Falcovitz, A second order Godunov type scheme for compressible fluid dynamics,  J. Comput. Phys., {\bf 55} (1984), 1--32.

\bibitem{BenArtzi-Falcovitz-2003} M. Ben-Artzi and J. Falcovitz, {\sl ``Generalized Riemann problems
in computational fluid dynamics''}, {\em Cambridge Monographs on
Applied and Computational Mathematics, Cambridge University
Press,} 11, 2003.

\bibitem{DSCS} M. Ben-Artzi, J. Falcovitz and J. Li,  The convergence of the GRP scheme, Disc.
 Cont. Dynam. Sys., {\bf 23} (2009), 1-27.

 \bibitem{MBA-Jiequan-Numer} M. Ben-Artzi and J. Li, Hyperbolic balance laws: Riemann invariants and
the generalized Riemann problem, Numer. Math.,  {\bf 106} (2007),
369--425.

\bibitem{Bouchut-Bou-Perthame}
F. Bouchut, Ch.  Bourdarias and B. Perthame,  A MUSCL method
satisfying all the numerical entropy inequalities,  Math.
Comp.,  {\bf 65} (1996), 1439--1461.

\bibitem{bressan} A. Bressan,
     The unique limit of the Glimm scheme,
   Arch. Rat. Mech. Anal.,
    {\bf 130} (1995), 205--230.

    \bibitem{chain} C. Chainais-Hillairet,  Second-order finite-volume schemes for a non-linear hyperbolic equation: Error estimate,  Math. Meth. Appl. Sci.,  {\bf 23} (2000), 467--490.

\bibitem{chen} G.-Q. Chen, M. Torres and W. Ziemer,  Gauss-Green theorem for weakly differentiable vector fields, sets of finite perimeter, and balance laws, Comm. Pure Appl. Math.,  {\bf 62} (2009), 242--304.



\bibitem{colella}
P. Colella and P. R. Woodward, The Piecewise Parabolic Method (PPM)
for gas dynamical simulations,  J. Comput. Phys., {\bf 54} (1984), 174--201.

\bibitem{dafermos} C. M. Dafermos, {\sl ``Hyperbolic Conservation Laws in Continuum Physics, Fourth Edition''}, {\em Grundlehren der Mathematischen Wissenschaften, vol. 325,\,Springer,} 2016.

\bibitem{despres}B. Despres,
     Lax theorem and finite volume schemes, Math. Comp.,
    {\bf 73} (2004), 1203--1234.

    \bibitem{ding-chen}
X. Ding, G.-Q. Chen and P. Luo, Convergence of the fractional step Lax-Friedrichs scheme and Godunov scheme for isentropic gas dynamics, Comm. Math.
Phys., {\bf 121} (1989), 63--84.

    \bibitem{diperna}R. J. DiPerna,
    Convergence of approximate solutions to conservation laws,
    Arch. Rat. Mech. Anal.,
    {\bf 82} (1983), 27--70.

    \bibitem{elling}V. Elling, A Lax-Wendroff type theorem for unstructured quasiuniform grids,  Math. Comp., {\bf 76} (2007), 251--272.



\bibitem{evans}L.C. Evans, {\sl ``Partial Differential Equations'', }
 American Mathematical Society,1998.

 \bibitem{eymard-gallouet}R. Eymard, T. Gallou\"{e}t and R. Herbin , {\sl ``Finite Volume Methods''},   in {\sl ``Handbook of Numerical Analysis, Vol. VII''}, Eds. P.G. Ciarlet and J.-L. Lions, North-Holland, 2000, 713--1020.

\bibitem{federer}H. Federer,
    {\sl ``Geometric Measure Theory '',} Springer-Verlag, 1969.

\bibitem{Tadmor1} U. S. Fjordholm, R. K\"{a}ppeli, S. Mishra and E. Tadmor, Construction of approximate entropy measure-valued solutions for hyperbolic systems of conservation laws,
 Found. Comput. Math.,  {\bf 17} (2017), 763--827.

 \bibitem{gallouet-herbin-latche} T. Gallou\"{e}t,  R. Herbin and J.-C. Latch\'{e}, On the weak consistency of finite volumes schemes for
    conservation laws on general meshes,
 SeMA Journal {\bf 76} (2019), 581--594.

 \bibitem{GodlewskiRaviart} E. Godlewski and P.-A.
Raviart, {\sl ``Hyperbolic Systems of Conservation Laws''},
Ellipses, 1991.

\bibitem{Godunov-59} S. K. Godunov, Finite  difference methods for numerical computations of discontinuous
solutions of equations of fluid dynamics,  Mat. Sb., {\bf 47} (1959),
271--295.

\bibitem{dumbser} C. R. Goetz and M. Dumbser, A novel solver for the generalized Riemann problem based on a simplified LeFloch-Raviart
    expansion and a local space-time discontinuous Galerkin formulation,
 J. Sci. Comput., {\bf 69} (2016), 805--840.

  \bibitem{goodman}  J. B. Goodman and R. J. Leveque, A geometric approach to high resolution TVD schemes, SIAM J. Numer. Anal.,  {\bf 25} (1988),  268--284.

      \bibitem{kroner} D. Kr\"{o}ner,  M.Rokyta and M. Wierse, A Lax-Wendroff type theorem for upwind finite volume schemes in 2-D,
East-West J. Numer. Math.,  {\bf 4} (1996), 279--292.


\bibitem{kru}S. N. Kru\v{z}kov, First order quasilinear
equations in several independent variables, Math. USSR Sbornik,
{\bf 10}(1970), 217--243.

\bibitem{lax1971}P. Lax,
 Shock waves and entropy,
in {\sl ``Contributions to nonlinear functional analysis''} (Proc. Sympos. Math.
Res. Center, Univ. Wisconsin ), 603--634, Academic
Press, New York, 1971.

\bibitem{lax-wendroff}P. Lax and B. Wendroff,
     Systems of conservation laws,
   Comm. Pure Appl. Math.,
    {\bf 13} (1960), 217--237.

    \bibitem{LeFloch-Liu-99} P. G. LeFloch and J.-G  Liu,  Generalized monotone schemes, discrete paths of extrema, and discrete entropy conditions,
 Math. Comput., {\bf 68} (1999), 1025--1055.

    \bibitem{Jiequan-Yue}J. Li and Y. Wang, Thermodynamical effects and high resolution methods for compressible fluid flows, J. Comput. Phys.,  {\bf 343} (2017), 340--354.

    \bibitem{Li-Yu}  T-T. Li and W-C. Yu, {\sl ``Boundary value problems for
quasilinear hyperbolic systems,''} Duke University Mathematics Series, V.  1985.


    \bibitem{Lions-Souganidis-95} P.-L. Lions and P.  E.  Souganidis, Convergence of MUSCL and filtered schemes for scalar conservation laws and
Hamilton-Jacobi equations, Numer. Math.,  {\bf 69} (1995),
441--470.

       \bibitem{LiuTP} T.-P. Liu,  Uniqueness of weak solutions of the Cauchy problem for general $2x2$ conservation laws, J. Diff. Eqs.,  {\bf 20} (1976),  369--388.

    \bibitem{Liu-Osher} X.-D. Liu, S. Osher and T. Chan,  Weighted essentially non-oscillatory schemes, J. Comput. Phys., {\bf 115} (1994),  200--212.

    \bibitem{morton} K. W. Morton and D. F. Mayers,
    {\sl ``Numerical Solution of Partial Differential Equations'',} Cambridge University Press, 2005.

    \bibitem{S.Osher-85} S. J. Osher, Convergence of generalized MUSCL schemes,
 SIAM J. Numer. Anal., {\bf 22} (1985), 947--961.

 \bibitem{qian-li}J. Qian, J. Li and S. Wang,  The generalized Riemann problems for compressible
 fluid flows: towards high order,  J. Comput. Phys., {\bf 259} (2014), 358--389.



\bibitem{richtmyer}R. D. Richtmyer and K. W. Morton,
    {\sl ``Difference Methods for Initial Value Problems '',} Wiley, 1967.
    \bibitem{seguin-vovelle}N. Seguin and J. Vovelle,
     Analysis and approximation of a scalar conservation law with a flux function with discontinuous coefficients,
    Math. Models Methods Appl. Sci.
    {\bf 13} (2003), 221--257.

     \bibitem{Shu} C.-W. Shu, High order WENO and DG methods for time-dependent convection-dominated PDEs: A brief survey of several recent developments,
  J. Comput.Phys.,  {\bf 316} (2016), 598--613.

  \bibitem{Tadmor} E. Tadmor, A review of numerical methods for nonlinear partial differential equations ,
 Bull. Amerc. Math. Soc.,  {\bf 49} (2012), 507--554.

\bibitem{Toro} E. F. Toro, Derivative Riemann solvers
for systems of conservation laws and  ADER methods,
J. Comput. Phys.,  {\bf 212} (2006), 150--165.






\bibitem{B.vanLeer-79}B. van Leer, Towards the ultimate conservative difference scheme.IV: A second order sequel to Godunov's method,
 J. Comput. Phys., {\bf 32}( 1979), 101--136.



    \bibitem{Vila} J. P. Vila, An analysis of a class of second-order accurate Godunov-type schemes,
 SIAM J. Numer. Anal., {\bf 26} (1989), 830--853.



 \end{thebibliography}
\end{document}